%% file: main.tex
\documentclass[12pt]{amsart}

\usepackage{mathtools}

\mathtoolsset{showonlyrefs=true}

\usepackage[a4paper,left=35mm,right=35mm,top=30mm,bottom=30mm,marginpar=25mm]{geometry}

\usepackage{amsmath}
\usepackage{amssymb}
\usepackage{amsthm}
\usepackage{eurosym}
\usepackage[dvips]{graphics}
\usepackage{graphicx}
\usepackage{epsfig}
\usepackage[colorlinks]{hyperref}
\usepackage[normalem]{ulem}
\usepackage{dsfont}
\usepackage[nocompress, space]{cite}
\usepackage{enumitem}   
\usepackage{cases}      
\usepackage{empheq}     

\allowdisplaybreaks

\usepackage{comment}

\hypersetup{
citecolor={cyan}
}
\usepackage{xcolor}

\usepackage{ifthen}
\makeindex

\renewcommand{\div}{\operatorname{div}}

\newcommand{\bx}{{\bf x}}

\newcommand{\bv}{{\bf v}}

\def\leq{\leqslant}
\def\geq{\geqslant}

\newcommand{\dive}{\operatorname{div}}

\newcommand{\osc}{\operatorname*{osc}} 
\newcommand{\avg}{\operatorname*{avg}} 

\numberwithin{equation}{section}

\newtheoremstyle{thmlemcorr}{10pt}{10pt}{\itshape}{}{\bfseries}{.}{10pt}{{\thmname{#1}\thmnumber{
#2}\thmnote{ (#3)}}}
\newtheoremstyle{thmlemcorr*}{10pt}{10pt}{\itshape}{}{\bfseries}{.}\newline{{\thmname{#1}\thmnumber{
#2}\thmnote{ (#3)}}}
\newtheoremstyle{defi}{10pt}{10pt}{\itshape}{}{\bfseries}{.}{10pt}{{\thmname{#1}\thmnumber{
#2}\thmnote{ (#3)}}}
\newtheoremstyle{remexample}{10pt}{10pt}{}{}{\bfseries}{.}{10pt}{{\thmname{#1}\thmnumber{
#2}\thmnote{ (#3)}}}
\newtheoremstyle{ass}{10pt}{10pt}{}{}{\bfseries}{.}{10pt}{{\thmname{#1}\thmnumber{
A#2}\thmnote{ (#3)}}}

\theoremstyle{thmlemcorr}
\newtheorem{theorem}{Theorem}
\numberwithin{theorem}{section}
\newtheorem{lemma}[theorem]{Lemma}
\newtheorem{corollary}[theorem]{Corollary}

\theoremstyle{thmlemcorr*}
\newtheorem{theorem*}{Theorem}
\newtheorem{lemma*}[theorem]{Lemma}
\newtheorem{corollary*}[theorem]{Corollary}
\newtheorem{proposition*}[theorem]{Proposition}
\newtheorem{problem*}[theorem]{Problem}
\newtheorem{conjecture*}[theorem]{Conjecture}

\theoremstyle{defi}
\newtheorem{definition}{Definition}

\newtheorem{hyps}{Assumption}          

\newtheorem{pde}{MFG System} 

\theoremstyle{remexample}
\newtheorem{remark}[theorem]{Remark}

\theoremstyle{ass}%
\theoremstyle{definition} 

\title[Regularity for Weak Solutions of First-Order Local MFG]{Regularity for Weak Solutions to First-Order Local Mean Field Games}
\author[A. M. Alharbi]{AbdulRahman M. Alharbi}
\address[A. M. Alharbi]{
        King Abdullah University of Science and Technology (KAUST), CEMSE Division, Thuwal 23955-6900. Saudi Arabia.}
\address[A. M. Alharbi]{
        The Islamic University of Madinah, Faculty of Science, Madinah, Saudi Arabia.}
\email{abdulrahman.alharbi@kaust.edu.sa, a.m.alharbi@iu.edu.sa}
\author[G. Di Fazio]{Giuseppe Di Fazio}
\address[G. Di Fazio]{
        University of Catania, Department of Mathematics and Computer Sciences, Catania, Italy.}
\email{giuseppe.difazio@unict.it}
\author[D. A. Gomes]{Diogo A. Gomes}
\address[D. A. Gomes]{
        King Abdullah University of Science and Technology (KAUST), CEMSE Division, Thuwal 23955-6900. Saudi Arabia.}
\email{diogo.gomes@kaust.edu.sa}
\author[M. Ucer]{Melih Ucer}
\address[M. Ucer]{
        King Abdullah University of Science and Technology (KAUST), CEMSE Division, Thuwal 23955-6900. Saudi Arabia.
        }
\email{melih.ucer@kaust.edu.sa}
\address[M. Ucer]{Ankara Yildirim Beyazit University (AYBU), Ankara. Turkey.}

\begin{document}
\begin{abstract}
We establish interior regularity results for first-order, stationary, local mean-field game (MFG) systems. Specifically, we study solutions of the coupled system consisting of a Hamilton--Jacobi--Bellman equation $H(x, Du, m) = 0$ and a transport equation $-\div(m D_pH(x, Du, m)) = 0$ in a domain $\Omega \subset \mathbb{R}^d$. 
Under suitable structural assumptions on the Hamiltonian $H$, without requiring monotonicity of the system, convexity of the Hamiltonian, separability in variables, or smoothness beyond basic continuity in $(p,m)$, we introduce a notion of weak solutions that allows the application of techniques from elliptic regularity theory. 
Our main contribution is to prove that the value function $u$  is locally Hölder continuous in $\Omega$. The proof leverages the connection between first-order MFG systems and quasilinear equations in divergence form, adapting classical techniques to handle the specific structure of MFG systems.
\end{abstract}
\maketitle


\input{problem}





\paragraph{\bf Funding: } The research reported in this paper was funded through King Abdullah University of Science and Technology (KAUST) baseline funds and KAUST OSR-CRG2021-4674.

\smallskip

\paragraph{\bf Acknowledgements} Most of this research was conducted during G. Di Fazio visit to KAUST, where he benefited from a supportive working atmosphere and an excellent research environment.

\bibliographystyle{abbrv}
\bibliography{mfgv2_nn}

\end{document}

%% file: problem.tex
\section{Introduction}

Mean-field games (MFG) provide a mathematical framework for modeling large populations of rational agents, each optimizing their individual objectives while being influenced by the collective behavior of the population. These systems have found widespread applications across multiple disciplines. In economics, they model price formation \cite{MR4215224, FT20, FTT20}, market dynamics \cite{cardaliaguetMeanFieldGame2018}, and wealth distribution \cite{GLasryLionsMoll}. Engineering applications include the analysis of traffic flow \cite{GomesMS19}, pedestrian dynamics \cite{Bagagiolo2022}, and smart grid coordination \cite{KizilkaleSM19, Kizilkale2014829}. In social sciences, MFGs have been instrumental in understanding opinion dynamics \cite{Bauso20143475, Bolouki2014556, Stella20132519}, segregation patterns \cite{AcCiMa}, and knowledge diffusion \cite{lucasmoll}. A key advantage of MFG theory is its ability to reduce the complexity of large-scale multi-agent systems to tractable systems of partial differential equations (PDEs), enabling both qualitative analysis and numerical simulation. This mathematical structure has sparked extensive research, leading to various theoretical developments and model extensions.
{ In this work, we focus on the interior regularity of \emph{first-order, stationary, local MFG systems}.}

{ Regularity properties in MFG systems are crucial both analytically and computationally. These
regularity results are essential for establishing solution existence and uniqueness and for understanding solution stability, and often ensure the convergence and accuracy of discretization schemes.
Our main result (Theorem \ref{thm:holder}) proves interior Hölder continuity for solutions of the following MFG system.}
\begin{pde} \label{main}
Let $\Omega\subset\mathbb{R}^d$ be an open and connected set, and let $H\colon \Omega\times \mathbb{R}^d \times \mathbb{R}^{+} \to \mathbb{R}$ be a measurable function. We assume that for almost all $x\in\Omega$, the mappings $(p,m) \mapsto H(x, p, m)$ and $(p,m) \mapsto D_pH(x, p, m)$ are continuous. Our MFG system consists of two coupled equations:
\begin{empheq}[left=\empheqlbrace]{alignat=2}
    & H(x, Du, m) = 0 &\text{in}\enskip \Omega, \label{eq:hjb} \\
    & -\dive(m D_pH(x, Du, m)) = 0 \qquad &\text{in}\enskip \Omega. \label{eq:fp}
\end{empheq}
\end{pde}
Equation \eqref{eq:hjb} is the \emph{Hamilton--Jacobi--Bellman} (HJB) equation, while \eqref{eq:fp} is the \emph{transport equation}. 

Section~\ref{sec:prelim} 
{introduces definitions, notation, and key assumptions on} the Hamiltonian $H$. In \textcolor{brown}{Sub}Section~\ref{subsec:assume}, we present a set of structural conditions on $H$ (Assumption~\ref{assume}) that enable us to develop a notion of weak solutions (Definition~\ref{def:solution}) and facilitate the proof of our main regularity results. A distinguishing feature of our framework is the remarkably mild nature of these assumptions. We do not require $H(x,p,m)$ to be separable; that is, $H$ of the form $H(x,p,m)=H_0(x,p)-g(m)$. 
{
In addition, we impose no monotonicity constraints or smoothness conditions beyond minimal continuity and differentiability requirements. Our analysis permits $H$ to be discontinuous in the spatial variable $x$ and requires neither differentiability in $m$ nor twice-differentiability in $p$. This level of generality significantly extends beyond, while still applying to standard examples such as
\begin{equation}
\label{standardH}
    H(x,p,m) = \frac{|p|^{\alpha}}{m^\tau} - m^{\beta},
\end{equation}
where $\alpha$, $\tau$, and $\beta$ are fixed real numbers satisfying \eqref{par1}. Our framework encompasses a broad class of Hamiltonians that may exhibit anisotropic behavior, non-monotonicity, or non-smoothness. 
Such flexibility allows our results to apply to problems where standard regularity assumptions may be too restrictive, yet where sufficient structure exists to establish meaningful solution regularity.
In particular, the conditions in \eqref{par1} do not imply Lions' monotonicity condition, derived in his lecture series \cite{LCDF}, which for the example in \eqref{standardH} reduces to $\tau\alpha^\prime \leq 4$.
}

\subsection{Background} \label{subss:background}

{MFGs emerged from two independent research streams.}
The term was introduced by Lasry and Lions in their seminal papers \cite{ll1,ll2,lasryMeanFieldGames2007a}, which laid the mathematical foundations of the field. Concurrently, Caines, Huang, and Malhamé \cite{HuangCainesMalhame06,Caines2} developed similar concepts from an engineering perspective, focusing on control-theoretic applications. Comprehensive presentations of the theory can be found in the course materials \cite{LCDF}, lecture notes \cite{cardaliaguet2018short}, and the survey \cite{Caines2017}. The framework has since been extended to encompass more complex scenarios, including systems with major and minor players \cite{Caines2, MA2020108774}, multi-population dynamics \cite{cirant}, and finite-state problems \cite{GMS, GMS2}, for example.

MFG System~\ref{main} describes the collective behavior of a large population of rational agents who act independently and competitively. 
{Each agent aims to minimize its cost, which depends on interactions} with both the population distribution and the environment, as they navigate toward potential exits from the domain. 
{This occurs non-cooperatively, with each agent optimizing its own objectives.}

Given a stationary probability distribution $m:\Omega\to \mathbb{R}^{+}_0$, consider a single agent who knows this distribution and seeks to minimize their total cost. For an initial position $x\in\Omega$, this cost is given by
\begin{equation} \label{eq:CostI}
\mathcal{I}[x,\mathbf{v}]:= \int_{0}^{T_E} L( \mathbf{x}(t),\mathbf{v}(t),m(\mathbf{x}(t))) dt
+\psi(\bx(T_E)) 1_{T_E<+\infty}.
\end{equation}
where $\mathbf{v}(\cdot) = \dot{\mathbf{x}}(\cdot)$ is the agent's control, $\mathbf{x}(\cdot)$ is their trajectory starting from $\mathbf{x}(0) = x$, 
$T_E$ is the exit time, $L$ represents the running cost, while $\psi$ is the exit cost.
The value function $u(x)$ represents the minimal cost achievable from position $x$:
\begin{equation}
u(x) = \inf_{\mathbf{v}(\cdot)} \mathcal{I}[x,\mathbf{v}].
\end{equation}
The exit time $T_E$ may be infinite, indicating that the agent might never exit the domain. When $T_E$ is finite, the exit cost $\psi:\partial \Omega\to \mathbb{R}$ determines the Dirichlet boundary conditions. Since our analysis focuses on interior regularity, the specific choice of boundary conditions (whether Dirichlet or state constraints) is not essential.

The instantaneous cost of an agent's actions is encoded in the function $L$, called the \emph{Lagrangian} or \emph{running cost}. This function captures the direct costs of the agent's actions and the indirect costs arising from interactions with other agents through the density $m$. The Lagrangian is mathematically related to the Hamiltonian $H$ through the Legendre transform:
\begin{equation}
\label{lt}
H(x,p,m):= \sup_{v} (-p\cdot v -L(x,v,m)).
\end{equation}
This transformation is crucial in converting the agent's optimization problem into the Hamilton--Jacobi--Bellman equation in our MFG system. Moreover, a Lagrangian can be recovered via the same transformation:
\begin{equation}
L(x,v,m):= \sup_{p} (-v\cdot p -H(x,p,m)).
\end{equation} 
Applying this to the model case \eqref{standardH}, we obtain:
\begin{equation}
   L(x,v,m) = m^{\frac{\tau}{\alpha-1}}\frac{|v|^{\alpha'}}{\alpha'} + m^{\beta}.
\end{equation}
Here, the effect of the exponents in \eqref{standardH} becomes more evident. Specifically, the exponent $\tau$ quantifies velocity-dependent congestion effects, representing the cost of moving at speed $v$ in density $m(x)$, while the exponent $\beta$ measures position-dependent congestion effects, indicating the cost of remaining at position $x$ in density $m(x)$. As such, \eqref{par1} implies that the position-dependent congestion effect is dominant{, at least if $v$ is bounded.}

Under the assumption of optimal behavior, each agent minimizes their individual cost as determined by the \emph{value function} $u(x)$, which represents the minimal cost-to-go from position~$x$. Optimal control theory establishes that $u$ satisfies the Hamilton--Jacobi--Bellman (HJB) equation \eqref{eq:hjb} in the viscosity sense. When the value function is sufficiently regular, the optimal control strategy takes the feedback form
\begin{equation}
\bv(t)= -D_pH(\bx(t), Du(\bx(t)), m(\bx(t))),
\end{equation}
where $D_pH$ represents the gradient of $H$ with respect to the momentum variable. For rational agents with a time-independent distribution $m$, the collective behavior of the population is described by the transport equation \eqref{eq:fp}, which determines the agent density. The coupling between these equations, the HJB equation describing individual optimization and the transport equation determining collective dynamics, constitutes our MFG system.

\subsection{Regularity theory in first-order MFGs}

The regularity theory for MFG systems takes different forms depending on whether the equations are elliptic, parabolic, or first-order. In elliptic and parabolic MFG systems, each equation can be analyzed separately using classical regularity theory. In contrast, first-order MFG systems behave as a single elliptic equation in the value function $u$. This structural property directly links the regularity theory of first-order MFGs and that of elliptic equations, suggesting that techniques from elliptic regularity theory should be effective in analyzing these systems.  For separable MFGs, this connection is particularly transparent through their variational formulation, see, e.g., \cite{GS}. Functionals of the form
\begin{equation}\label{eq:var-functional}
    \int_\Omega G(H(x,Du)) \, dx,
\end{equation}
where $G$ is strictly convex and increasing, yield Euler-Lagrange equations that can be written as MFG systems:
\begin{equation}\label{eq:mfg-system}
    \begin{cases}
        H(x,Du) = (G')^{-1}(m) & \text{in } \Omega,\\[1ex]
        -\dive(m D_pH(x, Du)) = 0 & \text{in } \Omega.
    \end{cases}
\end{equation}
A canonical example is the $\gamma$-Laplace equation, which corresponds to choosing $H(x,p) = \frac{2}{\gamma}|p|^{\gamma/2}$ and $G(z) = \frac{1}{2}z^2$: 
\begin{equation}\label{eq:p-laplace-mfg}
    \begin{cases}
        \dfrac{2}{\gamma} |D u|^{\gamma/2} = m & \text{in } \Omega,\\[1ex]
        -\dive(m |D u|^{\gamma/2-2} D u) = 0 & \text{in } \Omega.
    \end{cases}
\end{equation}

Variational methods have proven particularly fruitful in developing the theory of first-order MFGs. For example, \cite{Card1order} has proved the existence and uniqueness of weak solutions for locally coupled systems. This variational approach was subsequently extended in~\cite{cardaliaguetMeanFieldGames2014}, where weak solutions were characterized as minimizers of optimal control problems involving Hamilton--Jacobi--Bellman and transport equations under general structure conditions. The long-time behavior of such systems was investigated in~\cite{Cd1}, which connected the asymptotic analysis to ergodic MFG systems through weak KAM theory.
The connection to the $p$--Laplacian is also highlighted by the numerous results that use the variational formulation to establish the existence of weak solutions (see, for example, \cite{San16, PrSa2016, alharbi2023first}).
In particular, \cite{graber2017sobolev} uses duality methods to establish higher-order Sobolev regularity for suitable functions of the solutions. 

For first-order MFGs without variational structure, monotone operator methods provide a framework for constructing weak solutions. While Lasry and Lions used monotonicity ideas to 
establish uniqueness, their systematic use began with numerical applications~\cite{almulla2017two} and was subsequently developed for periodic stationary MFGs~\cite{FG2}, later encompassing Dirichlet boundary conditions~\cite{FGT1} and time-dependent problems~\cite{FeGoTa21}. Although these works establish existence under general conditions, they yield solutions with limited regularity. This paper addresses this gap by developing a higher regularity theory.

In contrast to second-order MFGs, where smooth solutions exist broadly~\cite{GM, PV15, 2016arXiv161107187C}, and comprehensive regularity estimates are available \cite{GPV}, regularity results for first-order systems are sparse. Notable exceptions include MFGs with logarithmic coupling~\cite{evans2003some, E2}, motivated by weak KAM theory and some class of first 
order MFGs that arise in the vanishing discount limit \cite{MR4175148}. 
However, 
the regularity theory for first-order MFGs is inherently limited, as demonstrated by explicit examples in~\cite{Gomes2016b,GNPr216}. 

The regularity theory for degenerate elliptic equations, which shares significant technical features with MFG systems, has seen substantial progress recently. Recent studies have established fundamental results for degenerate elliptic equations under minimal assumptions, including local boundedness and Harnack inequalities for equations with matrix weights \cite{20.500.11769_633169, 20.500.11769_585952}, boundary regularity for degenerate operators \cite{20.500.11769_489964}, and regularity results under natural growth conditions \cite{20.500.11769_518677}.
These advances in regularity theory, especially regarding equations with matrix-valued coefficients and weighted spaces \cite{20.500.11769_492793}, provide essential tools adaptable to analyzing MFG systems. {
Crucially, many existing regularity results for first-order MFGs hinge on the system possessing a specific variational structure or the Hamiltonian satisfying strong monotonicity conditions. This reliance limits their applicability to a broader class of problems. Our work directly confronts this limitation by establishing regularity under a framework of milder, more general assumptions, as detailed in the subsequent sections.
} Future research may explore the connection between degenerate elliptic equations and MFGs in the context of Morrey spaces, where the degeneracy structure of the equations could play a crucial role in determining solution properties.

\subsection{Contributions and main results}
Our primary contribution is establishing interior regularity for solutions of MFG System~\ref{main}. Specifically, we prove Hölder continuity of the value function $u$ (Theorem~\ref{thm:holder}) and derive associated Harnack inequalities. 
{
We study weak solutions $(m,u)$. In this pair, $m$ is a non-negative, locally integrable function, and $u$ belongs to a Sobolev space whose specific properties are detailed later (see Definition~\ref{def:solution} and \eqref{eq:mu-space}). These solutions satisfy equation \eqref{eq:hjb} in the strong sense and equation \eqref{eq:fp} in the distributional sense, according to the precise conditions outlined in Definition~\ref{def:solution}.
}
Our analysis relies on a carefully chosen set of structural assumptions on the Hamiltonian $H$ (Assumption~\ref{assume}, detailed in Section~\ref{subsec:assume}). These conditions ensure both the well-definedness of the distributional formulation in \eqref{eq:fp} and the validity of key estimates needed in our proofs. 
{
Section~\ref{sec:holder} develops the complete regularity theory, presenting detailed proofs of Harnack-type inequalities that culminate in the following Hölder continuity result (Theorem~\ref{thm:holder}).
}
\begin{theorem}[Hölder continuity of the value function]\label{thm:holder}
Consider MFG System~\ref{main} and suppose that Assumption~\ref{assume} holds. Then, any value function $u$ arising from a solution pair $(m,u)$ in the sense of Definition~\ref{def:solution} is locally Hölder continuous in $\Omega$. 
\end{theorem}

A prototypical example of a Hamiltonian satisfying Assumption~\ref{assume} is given in \eqref{standardH}, while a more general form for such a Hamiltonian is given in~\eqref{eq:h-typic}.

Our regularity theory emerges from viewing equation \eqref{eq:fp} as a quasi-linear divergence-form second-order PDE in the variable $u$. This perspective connects our work to classical regularity results for elliptic equations. Specifically, we build upon the foundational papers  \cite{serrin64} and \cite{trudinger}, who established $C^{0,\alpha}$ estimates for equations with $p$-Laplacian-type growth and coercivity conditions, extending Moser's work \cite{moser} on linear elliptic equations. In our framework, equation \eqref{eq:hjb} combined with Assumption~\ref{assume} provides the necessary growth and coercivity estimates (Lemma \ref{lem:pointwise}), allowing us to apply these classical techniques to equation \eqref{eq:fp}.
The key insight is that while equation \eqref{eq:hjb} does not uniquely determine $m$ in terms of $(x,Du)$, it provides sufficient bounds on $m$ to enable the application of the techniques developed in \cite{serrin64, trudinger, moser}. This approach succeeds despite our system not precisely fitting their framework.

\section{Preliminaries}
\label{sec:prelim}
{
This section details essential preliminaries: 
we first introduce notation, then discuss assumptions on the Hamiltonian $H$ in MFG System~\ref{main}, and finally define weak solutions.”  
}
To simplify the analysis in the next section, we also derive preliminary estimates.

\subsection{Notation}\label{subsec:note}

{Differential operators without subscripts (e.g., $Du$, $\dive F(x)$) denote differentiation with respect to the spatial variable $x$.  } Explicit subscripts indicate the differentiation with respect to other variables. In particular, in the expression $D_pH(x,p,m)$, the subscript $p$ consistently refers to the middle variable, even when we substitute $Du$ for $p$.

Secondly, an inequality involving a constant $C$ must be interpreted as: \emph{there exists a sufficiently large positive value $C$, which is independent of all free variables, such that the inequality is satisfied}. In particular, the form of any such inequality must be so that whenever $C$ is replaced by a greater value, the inequality remains valid. Typically, we write inequalities of the form
\[C^{-1}a - Cb \leq c \leq Cd - C^{-1}e,\]
where $a,b,c,d,e$ are various variable expressions such that $a, b, d, e > 0$. Note that this interpretation implies that the value of $C$ may change from one line to another. When the constant $C$ in such an inequality depends on one of the free variables, it is indicated with a subscript.

 \subsection{Assumptions and Consequences} \label{subsec:assume}

{
In this subsection, we describe the 
 structural assumptions on the Hamiltonian, $H\colon \Omega\times \mathbb{R}^d \times \mathbb{R}^{+} \to \mathbb{R}$. They provide the necessary control and algebraic properties, primarily through power-like asymptotic behavior, that allow us to derive the the estimates essential for the elliptic regularity techniques employed in Section~\ref{sec:holder} to prove Hölder continuity. These conditions are designed to be as mild as possible while still yielding the desired regularity.}
 
{
Before detailing these assumptions, we reiterate a key point regarding the domain of $H$.
Since $H(x,p,m)$ is only defined for $m>0$,
the expressions $H(x,p,m)$ and $D_pH(x,p,m)$, appearing in Assumption~\ref{assume} below, are defined for all $m>0$ and for all $x \in \Omega$ and $p \in \mathbb{R}^d$. }
Assumption~\ref{assume} amounts to imposing a certain power-like asymptotic behavior on $H(x,p,m)$ and $D_pH(x,p,m)$, as $|p|$ tends to infinity, $m$ tends to infinity, or $m$ tends to $0$. A standard Hamiltonian that satisfies such assumptions is
\eqref{standardH}. 
 However, as discussed later in the section, the assumptions are substantially more general than this standard example.

In the formulation of Assumption~\ref{assume}, we have three key real parameters $\alpha$, $\tau$, and $\beta$ for which we assume
\begin{equation}
\label{par1}
    \alpha >1, \qquad 0\leq \tau < 1, \qquad \beta > \frac{\tau}{\alpha-1}.
\end{equation}
In addition, we denote
\begin{equation}
\label{par2}
\delta := \frac{\beta+\tau}{\alpha}.
\end{equation}
Note that the third inequality in \eqref{par1} between $\alpha$, $\tau$, and $\beta$ is equivalent to $\beta > \delta$. A given $H$ can satisfy Assumption~\ref{assume} only for a unique triple $(\alpha, \tau, \beta)$.

\begin{hyps}\label{assume}
In the following estimates, $C$ is a sufficiently large constant that is independent of the free variables $x$, $p$, and $m$.
In particular, the bounds are uniform over $\Omega$. 
Let $\alpha$, $\tau$, $\beta$, and $\delta$ be real parameters satisfying 
\eqref{par1} and \eqref{par2}, and let $\epsilon > 0$
be such that $\beta-\delta>\epsilon$.
\begin{enumerate}[label=\textbf{A.\arabic*}]\addtocounter{enumi}{-1}

\item We assume that the Hamiltonian, $H$, satisfies the assumptions outlined in the statement of MFG System \ref{main}; that is, $H$ is a measurable function such that $(p,m) \mapsto H(x, p, m)$ and $(p,m) \mapsto D_pH(x, p, m)$ are continuous for almost all $x\in\Omega$.   
    \item\label{assume:DpH.growth} The magnitude of the derivative $D_pH(x,p,m)$ is bounded by a power function in $|p|$ and $m$, as follows:
    \begin{align}
        |D_pH(x,p,m)| \leq C\left(\frac{1}{m}+\frac{1}{m^\tau}\right)|p|^{\alpha-1} + C(m^{\beta-\delta}+1/m).
    \end{align}
    
    \item\label{assume:DpH.coerce} The component of the derivative $D_pH(x,p,m)$ in the direction of $p$ satisfies the coercivity condition
    \begin{align}
        D_pH(x,p,m)\cdot p \geq C^{-1}\left(\frac{1}{m^\tau+1}\right)|p|^{\alpha} - C(m^{\beta-\delta-\epsilon}+1)|p|.
    \end{align}
    \item\label{assume:H.m-coerce} The Hamiltonian $H$ at fixed $p=0$ satisfies the inequalities
    \begin{alignat*}{2}
         H(x,0,m) & \geq -C(m^{\beta}+1) \qquad && \text{for} \quad m > 0, \\
         H(x,0,m) & \leq -C^{-1}m^\beta && \text{for} \quad m \geq C.
    \end{alignat*}
\end{enumerate}
\end{hyps}

\begin{remark}
\label{R1}
    We observe that Assumption~\ref{assume:DpH.coerce} implies
    \[
D_pH(x,p,m)\cdot p \geq C^{-1}\left(\frac{1}{m^\tau+1}\right)|p|^{\alpha} - C(m^{\beta-\delta-\Tilde{\epsilon}}+1)|p|,
    \]
    for any $0\leq \Tilde{\epsilon}\leq \epsilon$, by taking a (possibly) larger constant $C$. 
\end{remark}

{
\begin{remark}
    Although Assumption~\ref{assume} states that constants are uniform in $\Omega$, local uniformity suffices for interior regularity results; that is, the bounds must hold uniformly on any open set $U \Subset \Omega$ (i.e., $U$ such that its closure $\bar{U}$ is a compact subset of $\Omega$).
\end{remark}
}

The following lemma derives further structural properties of $H$ based on Assumption~\ref{assume}, which gives an analog of the lower and upper bounds in  \ref{assume:H.m-coerce} that  hold for all values of $p$.
\begin{lemma}\label{lem:h.bounds}
    Let $H\colon \Omega\times \mathbb{R}^d \times \mathbb{R}^{+} \to \mathbb{R}$ satisfy Assumption~\ref{assume}. Then
    \begin{alignat}{2}
    & H(x,p,m) \geq C^{-1}|p|^{\alpha}/(m^\tau+1) -Cm^{\beta} -C, && \qquad \text{for}\quad m > 0, \label{eq:h.geq} \\ 
    & H(x,p,m) \leq C|p|^{\alpha}/m^{\tau} -C^{-1}m^{\beta}, && \qquad \text{for}\quad m\geq C. \label{eq:h.leq}
\end{alignat}
\end{lemma}
\begin{proof}
    From the fundamental theorem of calculus, we have that
    \begin{equation}\label{eq:h.DpH.formula}
        H(x,p,m) = H(x,0,m) + \int_0^1 D_pH(x,pt,m)\cdot p \, dt.
    \end{equation}
    Assumption~\ref{assume:DpH.coerce} and Remark \ref{R1} imply
    \[D_pH(x,pt,m)\cdot p \geq C^{-1}\left(\frac{1}{m^\tau+1}\right)|p|^{\alpha}t^{\alpha-1} - C(m^{\beta-\delta}+1)|p|.\]
    Hence,
    \[\int_0^1 D_pH(x,pt,m)\cdot p  \, dt\geq \frac{1}{\alpha}C^{-1}\left(\frac{1}{m^\tau+1}\right)|p|^{\alpha} - C(m^{\beta-\delta}+1)|p|.\]
    Incorporating this estimate in \eqref{eq:h.DpH.formula} and invoking assumption \ref{assume:H.m-coerce}, we obtain
    \begin{equation}\label{eq:Lem2.1Pfeq2}
        H(x,p,m) \geq -C(m^{\beta}+1) + C^{-1}|p|^{\alpha}/(m^\tau+1) - C|p|(m^{\beta-\delta}+1). 
    \end{equation}
    Then, using Young's inequality, we have
    \[|p|(m^{\beta-\delta}+1) \leq \sigma\frac{|p|^\alpha}{m^\tau+1} + C_{\sigma}\left((m^\tau+1)(m^{\beta-\delta}+1)^{\alpha}\right)^{1/(\alpha-1)}\]
    for any $\sigma > 0$. Moreover, noting that $ (\tau+(\beta-\delta)\alpha)/(\alpha-1) = \beta$, standard estimates give 
    \[\left((m^\tau+1)(m^{\beta-\delta}+1)^{\alpha}\right)^{1/(\alpha-1)} \leq C(m^\beta + 1).\]
    Incorporating the last two inequalities with suitably small $\sigma$ into \eqref{eq:Lem2.1Pfeq2}, we obtain 
    the estimate \eqref{eq:h.geq}. 

    Next, assuming $m\geq C$ (that is, $m$ is sufficiently large), \ref{assume:DpH.growth} implies
    \[|D_pH(x,pt,m)| \leq C\frac{|p|^{\alpha-1}}{m^\tau}t^{\alpha-1} + Cm^{\beta-\delta}.\]
    Hence,
    \[\int_0^1 D_pH(x,pt,m)\cdot p  \, dt\leq |p|\int_0^1 |D_pH(x,pt,m)| \, dt \leq C\frac{|p|^{\alpha}}{m^\tau} + Cm^{\beta-\delta}|p|.\]
    Combining this estimate with $H(x,0,m)\leq -C^{-1}m^\beta$ of \ref{assume:H.m-coerce} in the expression \eqref{eq:h.DpH.formula}, we get
    \[H(x,p,m) \leq C\frac{|p|^{\alpha}}{m^\tau} - C^{-1}m^\beta + C|p|m^{\beta-\delta}.\]
    Then, we conclude \eqref{eq:h.leq} by applying Young's inequality as before. 
\end{proof}

\subsubsection{Discussion on the Assumptions}\label{par:discussion}

Let us examine the qualitative implications of Assumption~\ref{assume}, focusing on the behavior of $H$ at a fixed $x\in \Omega$ as $p$ and $m$ vary.

Assumptions~\ref{assume:DpH.growth} and~\ref{assume:DpH.coerce}, at fixed $m$, characterize the asymptotics of $p\mapsto D_pH(x,p,m)$ for large $|p|$.  The vector field's radial component scales as $|p|^{\alpha-1}$, while the tangential component is controlled from above by the same quantity. { This implies that the tangential component does not dominate the radial component, ensuring a steady coercive growth of the Hamiltonian as $|p|\to \infty$.}
Such asymptotic behavior with respect to $p$ uniquely determines
$\alpha$ and typically corresponds to Hamiltonians of the form:
\[H(x,p,m) \sim |p|^{\alpha} + \text{lower order terms},\]
a standard assumption in MFG.

Furthermore, Assumptions~\ref{assume:DpH.growth} and~\ref{assume:DpH.coerce} govern how the ratio between $D_pH(x,p,m)$ and $|p|^{\alpha-1}$ scales with $m$. As $m\to\infty$, this ratio scales as $m^{-\tau}$ for a unique $\tau\geq 0$, yielding the typical form:
\[H(x,p,m) \sim \frac{|p|^{\alpha}}{m^\tau} + \text{lower order terms}.\]

As $m\to 0$, the ratio is bounded below by 1 and above by $1/m$. As such, $m^{-\tau}$ scaling with $\tau\in[0,1)$ remains possible, { as this scaling arises in congestion models}.
Moreover, this bound ensures control of $mD_pH(x,p,m)$ in \eqref{eq:fp}. The offset terms follow similar asymptotics: $O(m^{\beta-\delta})$ and $o(m^{\beta-\delta})$ as $m\to\infty$, and $O(1/m)$ and $O(1)$ as $m\to 0$.

Assumption~\ref{assume:H.m-coerce} characterizes the behavior of $m\mapsto H(x,p,m)$ at $p=0$. By Lemma~\ref{lem:h.bounds}, this behavior is uniform in $p$. For large $m$, $H$ is negative and scales as $m^\beta$ for a unique $\beta$, while as $m\to 0$, $H$ is only bounded below.

Combining all assumptions, a prototypical Hamiltonian takes the form:
\begin{equation}\label{eq:h-typic}
   H(x,p,m) = a(x)\frac{|p|^{\alpha}}{m^\tau} - b(x)m^{\beta} + \text{lower order terms,}
\end{equation}
where $a,b \in L^\infty(\Omega)$ have positive essential infima.

For a precise notion of lower-order terms in \eqref{eq:h-typic}, consider terms of form
\begin{equation}
\label{eq:loworder}
    c(x)f(m)|p|^{\theta}
\end{equation}
where $c\in L^\infty(\Omega)$ and $f\colon \mathbb{R}^{+}\to \mathbb{R}$ is continuous.
{ Such terms can be added to the Hamiltonian \eqref{eq:h-typic} if $\theta$ and the growth of $f$ (as $m\to\infty$ and $m\to 0$) satisfy suitable conditions.}
 Either of the following conditions suffices:
\begin{enumerate}
   \item \label{item:strong-low-order} $\theta = 0$ or $1 < \theta < \alpha$, with
   \begin{alignat}{2}
       & |f(m)| = o(m^{\beta-\delta\theta}) \qquad\quad &&\text{as}\quad m\to\infty,\\
       & |f(m)| = O(1) &&\text{as}\quad m\to 0.
   \end{alignat}
   \item \label{item:signed-low-order} $c(x)\geq 0$, $f(m)\geq 0$, and $\theta = 0$ or $1 < \theta \leq \alpha$, with
   \begin{alignat}{3}
       & f(m) = o(m^\beta) \qquad\quad &&\text{as}\quad m\to\infty, \qquad &&\text{if}\quad \theta=0,\\
       & f(m) = O(m^{\beta-\delta\theta}) \qquad\quad &&\text{as}\quad m\to \infty, &&\text{if}\quad 1 < \theta\leq \alpha,\\
       & f(m) = O(1/m) &&\text{as}\quad m\to 0, &&\text{if}\quad 1 < \theta\leq \alpha.
   \end{alignat}
\end{enumerate}
{
Concrete examples of such lower-order terms include:}
\begin{enumerate}[label=(\alph*)]
   \item Bounded potentials $V(x)\in L^{\infty}(\Omega)$, satisfying item~\ref{item:strong-low-order} with $\theta=0$ and $f(m) = 1$.
   
   \item Functions $f(m)$ that are $o(m^\beta)$ as $m\to\infty$ and bounded below as $m\to 0$. Here, $f^+$ satisfies item~\ref{item:signed-low-order} and $f^-$ satisfies item~\ref{item:strong-low-order} with $c(x)=1$, $\theta=0$ (e.g., $f(m)=-\ln{m}$).
   
   \item Terms of form
   \[|p|^{\tilde{\alpha}}/m^{\tilde{\tau}}\]
   with $1 < \tilde{\alpha} \leq \alpha$ and $\delta\tilde{\alpha}-\beta \leq \tilde{\tau} \leq 1$, satisfy item~\ref{item:signed-low-order} with $c(x)=1$, $\theta=\tilde{\alpha}$, and $f(m) = m^{-\tilde{\tau}}$.
\end{enumerate}

\subsection{The Definition of Weak Solution}

This subsection introduces the weak solution concept for MFG System~\ref{main}. The Hamiltonian $H$ is assumed to satisfy Assumption~\ref{assume}, characterized by parameters $\alpha$, $\beta$, $\tau$, and a sufficiently small $\epsilon$.

We begin by defining
\[\gamma := \frac{\beta+1}{\delta} = \frac{\beta+1}{\beta+\tau}\alpha,\]
and we observe the identities
\begin{equation}\label{eq:parameters-identity}
   \delta = \frac{\beta+\tau}{\alpha} = \frac{\beta+1}{\gamma} = \frac{1-\tau}{\gamma - \alpha} = \frac{\beta+1-\delta}{\gamma-1}.
\end{equation}
Noting $\gamma > \alpha > 1$, we seek weak solutions $(m,u)$ with
\begin{equation}\label{eq:mu-space}
   m\in L^{\beta+1}_{\text{loc}}(\Omega), \quad m\geq 0, \qquad\text{and}\qquad u \in W^{1,\gamma}_{\text{loc}}(\Omega).
\end{equation}

To motivate the definition, consider first the case where $m>0$ a.e.~in $\Omega$. We define
\begin{equation}
   h(x):= H(x, Du(x), m(x)) \quad\text{and}\quad j(x):= m(x)D_pH(x, Du(x), m(x)).
\end{equation}
Then, according to Definition~\ref{def:solution} below,
$(m,u)$ solves~\eqref{eq:hjb} if $h = 0$ a.e.~in~$\Omega$. { Moreover, noting that $j$ is locally integrable (see Lemma~\ref{lem:test}), $(m,u)$ solves~\eqref{eq:fp} if $-\dive j = 0$ in the distributional sense.}
The strict positivity of $m$ is crucial because of the domain of $H$.

Definition~\ref{def:solution} generalizes this concept to allow $m\geq 0$, providing a weaker solution framework compatible with prior existence results. The key challenge is extending the continuous maps
\begin{equation}\label{eq:the-2-maps}
   (p,m)\mapsto H(x,p,m) \qquad\text{and}\qquad (p,m)\mapsto mD_pH(x,p,m)
\end{equation}
to handle pairs $(p,0)$ at any $x\in \Omega$. { We address this challenge by extending these maps in a set-valued sense and redefining $h$ and $j$ with inclusion instead of equality.

Specifically, let $F$ be a continuous map from $\mathbb{R}^d\times \mathbb{R}^{+}$ to $\mathbb{R}^{\tilde{d}}$ for some ${\tilde{d}}$. We take the closure of the graph of $F$ in $\mathbb{R}^d\times \mathbb{R}^{+}\times \mathbb{R}^{\tilde{d}}$ to obtain the graph of a set-valued map $\mathbb{R}^d\times [0,\infty)\to \mathbb{R}^{\tilde{d}}$, which we still denote by $F$. Explicitly, $F(p,m)$ is a singleton for $m>0$ and
\[
\begin{array}{l}
F(p,0) = \{ z \in \mathbb{R}^{\tilde{d}} \colon \exists \, \{(p_n,m_n)\}_{n\geq0} \enskip \text{such that} \\[.5em]
\phantom{F(p,0) = \{z \in } \lim_{n\to \infty} (p_n,m_n) = (p,0) \enskip
\text{and} \enskip  \lim_{n\to \infty} F(p_n,m_n) = z \}.
\end{array}
\]
With this in mind, we redefine $h$ and $j$ as measurable functions over~$\Omega$, defined a.e., satisfying
\begin{equation}\label{eq:the-2-variables.hj}
   h(x) \in H(x, Du(x), m(x)) \quad\text{and}\quad j(x) \in m(x)D_pH(x, Du(x), m(x)),
\end{equation}
where the right-hand sides denote the set-valued extensions of the maps in~\eqref{eq:the-2-maps} evaluated at $(Du(x), m(x))$. While $h$ and $j$ may be non-unique when $m(x)=0$ on a set of positive measure, this non-uniqueness is immaterial as we consider all possible $h$ and $j$. Also note that, in defining $j$, one does not first extend the function $D_pH(x,\cdot,\cdot)$ and then multiply by $m$, but one does it in the opposite order, so that one may have $j(x)\neq 0$ even in the set $\{x\colon m(x)=0\}$.
}



{
We proceed to show that any such $j$ is locally integrable, allowing \eqref{eq:fp} to be interpreted in the distributional sense. Then, we present the precise weak solution concept in Definition~\ref{def:solution}.}

\begin{lemma}\label{lem:test}
Consider the setting of MFG System \ref{main}, suppose that Assumption~\ref{assume} holds.
Let $(m,u)$ be as in \eqref{eq:mu-space} and consider any $j$ satisfying \eqref{eq:the-2-variables.hj}. Then $j\in L^{\gamma'}_{\text{loc}}(\Omega)$ where $\gamma'$ denotes the conjugate exponent $\gamma/(\gamma-1)$.
\end{lemma}
\begin{proof}
For almost all $x\in\Omega$,  Assumption~\ref{assume:DpH.growth} gives
\begin{equation}
    |mD_pH(x,p,m)| \leq C(1 + m^{1-\tau})|p|^{\alpha-1} + C(m^{\beta+1-\delta}+1)
\end{equation}
for $p\in\mathbb{R}^d$ and $m>0$. Since the right-hand side extends continuously to $m=0$, the inequality holds for $p\in\mathbb{R}^d$ and $m\geq 0$, with the left-hand side interpreted as any possible value of the continuous function $(p,m)\mapsto mD_pH(x,p,m)$. Consequently, we have
\begin{equation}\label{eq:j.upper.bound}
    |j| \leq C(1 + m^{1-\tau})|Du|^{\alpha-1} + C(m^{\beta+1-\delta}+1).
\end{equation}
Both terms on the right-hand side are in $L^{\gamma'}_{\text{loc}}(\Omega)$ because, \eqref{eq:parameters-identity} gives
\[\frac{1-\tau}{\beta+1} = \frac{\gamma-\alpha}{\gamma}, \qquad\text{hence}\quad \frac{1-\tau}{\beta+1} + \frac{\alpha-1}{\gamma} = \frac{1}{\gamma'},\]
and
\[\frac{\beta+1-\delta}{\beta+1} = \frac{1}{\gamma'}.\qedhere\]
\end{proof}

\begin{definition}
\label{def:solution}
Consider the setting of MFG System \ref{main}, suppose that Assumption~\ref{assume} holds. A \emph{solution} is a pair $(m,u)$ satisfying \eqref{eq:mu-space} and solving \eqref{eq:hjb} and \eqref{eq:fp} in the following sense:
\begin{enumerate}
    \item\label{def:hjb} We say that $(m,u)$ \emph{solves} \eqref{eq:hjb} if
        \[\begin{alignedat}{2}
            &h\leq 0 && \text{a.e.~in } \enskip\Omega,\\
            &h = 0 \qquad && \text{a.e.~in }\enskip \{x\in\Omega\colon m(x)>0\},
        \end{alignedat}\]
    for some $h$ satisfying \eqref{eq:the-2-variables.hj}.

    \item\label{def:fp} We say that $(m,u)$ \emph{solves} \eqref{eq:fp} if
    \[\int_{\Omega} j\cdot D\phi = 0,\]
    for some $j$ satisfying \eqref{eq:the-2-variables.hj}, for all smooth $\phi$ that are compactly supported in $\Omega$. Note that, since $j\in L^{\gamma'}_{\text{loc}}(\Omega)$ by Lemma \ref{lem:test}, the equality also holds for any $\phi\in W^{1,\gamma}(\Omega)$  compactly supported in $\Omega$, by the density of smooth functions.
\end{enumerate}
\end{definition}

We note that the condition $h \leq 0$ (rather than $h = 0$) is both sufficient for our regularity analysis and more naturally aligned with existence results where $u$ is only a subsolution in the region $m=0$.

\subsection{Preliminary Estimates}

In this subsection, we establish some estimates for the solutions $(m,u)$ of MFG System \ref{main} in the sense of Definition~\ref{def:solution}. We begin with a pointwise estimate based on \eqref{eq:hjb}.

\begin{lemma}\label{lem:pointwise}
Consider the setting of MFG System \ref{main}, suppose that Assumption~\ref{assume} holds.
    Let $(m,u)$ be a solution of \eqref{eq:hjb} in the sense of the item~\ref{def:hjb} of Definition~\ref{def:solution}. Then, we have
    \begin{alignat}{1}
        & m \leq C|Du|^{1/\delta} + C, \label{eq:m.leq}\\
        & m \geq C^{-1}|Du|^{1/\delta} - C \label{eq:m.geq}
    \end{alignat}
    a.e.~in~$\Omega$. Consequently, for any $j$ satisfying \eqref{eq:the-2-variables.hj}, we have
    \begin{align}
        |j| & \leq C|Du|^{\gamma-1} + C, \label{eq:mDpH_bound}\\
        j\cdot Du & \geq C^{-1}|Du|^{\gamma} - C \label{eq:mDpHdotp_bound}
    \end{align}
    a.e.~in~$\Omega$.
\end{lemma}
\begin{proof}
Fix $x \in \Omega$, and denote $Du := Du(x)$ and $m := m(x)$.
To prove \eqref{eq:m.leq}, we may assume $m > C$ for the constant in Lemma~\ref{lem:h.bounds}, as otherwise the inequality is trivial. Under this assumption, $H(x, Du, m) = h = 0$, and by \eqref{eq:h.leq} of Lemma~\ref{lem:h.bounds}, we have
\[C\frac{|Du|^{\alpha}}{m^\tau} - C^{-1}m^\beta \geq 0,\]
from which we immediately conclude \eqref{eq:m.leq}.

To establish \eqref{eq:m.geq}, we observe that the right-hand side of inequality \eqref{eq:h.geq} from Lemma~\ref{lem:h.bounds} extends continuously to $m = 0$. { Therefore, this inequality remains valid for all $p \in \mathbb{R}^d$ and $m \geq 0$, where the left-hand side is interpreted as any value in $H(x,p,m)$ according to the set-valued extension of $H(x,\cdot,\cdot)$}. Specifically, we have
\[h\geq C^{-1}|Du|^{\alpha}/(m^\tau+1) -Cm^{\beta} -C,\]
for any $h$ satisfying \eqref{eq:the-2-variables.hj}.
Since, by definition, there exists such an $h$ satisfying $h \leq 0$, we obtain
\[
   C^{-1}\frac{|Du|^{\alpha}}{m^{\tau}+1} \leq Cm^\beta + C,
\]
which can be rearranged and combined with standard estimates to yield \eqref{eq:m.geq}.

Next, using \eqref{eq:parameters-identity}, \eqref{eq:m.leq}, and \eqref{eq:m.geq} with standard estimates, we get
\begin{equation}\label{eq:m-to-1minustau-bounds}
    C^{-1}|Du|^{\gamma-\alpha} - C \leq m^{1-\tau} \leq C|Du|^{\gamma-\alpha} + C,
\end{equation}
and
\begin{equation}\label{eq:m-to-betaplus1minusdelta-bound}
    m^{\beta+1-\delta} \leq C|Du|^{\gamma-1} + C.
\end{equation}
Using these two results while observing that \eqref{eq:j.upper.bound} still holds, we find
\[|j| \leq C(1+m^{1-\tau})|Du|^{\alpha-1} + C(m^{\beta+1-\delta}+1) \leq C|Du|^{\gamma-1} + C,\]
concluding \eqref{eq:mDpH_bound}. Moreover, in the same way that \eqref{eq:j.upper.bound} was derived from \ref{assume:DpH.growth}, we can use \ref{assume:DpH.coerce} to achieve
\[j\cdot Du \geq C^{-1}\frac{m}{m^\tau+1}|Du|^{\alpha} - C(m^{\beta+1-\delta-\epsilon}+m)|Du|.\]
Noting the simple inequalities
\[\frac{m}{m^\tau+1} \geq \frac{m^{1-\tau}-1}{2} \qquad\text{and}\qquad m^{\beta+1-\delta-\epsilon}+m \leq 2m^{\beta+1-\delta-\epsilon} + 1,\]
we can simplify the last estimate as
\begin{equation}\label{eq:j.lower.bound}
    j\cdot Du \geq C^{-1}m^{1-\tau}|Du|^{\alpha} - Cm^{\beta+1-\delta-\epsilon}|Du| - C(|Du|^{\alpha}+|Du|).
\end{equation}
Now, using Young's inequality together with \eqref{eq:m-to-betaplus1minusdelta-bound}, we get
\[m^{\beta+1-\delta-\epsilon} \leq (\sigma/C) m^{\beta+1-\delta} + C_\sigma \leq \sigma|Du|^{\gamma-1} + C_\sigma\]
for any $\sigma > 0$. Then, using this with a suitably small $\sigma$ and \eqref{eq:m-to-1minustau-bounds} in the estimate \eqref{eq:j.lower.bound}, we find
\[j\cdot Du\geq C^{-1}|Du|^{\gamma} - C|Du|^{\alpha} - C|Du|.\]
We then conclude \eqref{eq:mDpHdotp_bound} with another application of Young's inequality.
\end{proof}

We now establish a Caccioppoli-type integral estimate for solutions of MFG System~\ref{main}. This estimate, which involves only the variable $u$, is the foundation for all subsequent estimates in this paper and, ultimately, the basis for the proof of Theorem~\ref{thm:holder}.

\begin{lemma}\label{lem:basic-integral-estimate}
Consider the setting of MFG System \ref{main}, suppose that Assumption~\ref{assume} holds.
    Let $(m,u)$ be a solution in the sense of Definition~\ref{def:solution}. Let $\xi$ be a smooth, compactly supported, non-negative function over $\Omega$, and let $f\colon\mathbb{R}\to\mathbb{R}$ be a $C^1$ function such that $f'$ is strictly positive and globally bounded.
    Then
    \[\int |\xi Du|^{\gamma}f'(u) \leq C\int |D\xi|^{\gamma}\frac{|f(u)|^\gamma}{f'(u)^{\gamma-1}} + C\int |\xi|^{\gamma}f'(u).\]
    As usual, $C$ depends only on the data in~\ref{assume}, so it is independent of $(m,u)$, $\xi$, and $f$.
\end{lemma}
\begin{proof}
    Observe that $f(u)\in W^{1,\gamma}_{\text{loc}}(\Omega)$ since $Df(u) = f'(u)Du$ and $f'(u)$ is bounded. Therefore, the integral equality in the item~\ref{def:fp} of Definition~\ref{def:solution} holds for $\phi := \xi^\gamma f(u)$. Thus we have
    \[\int (j\cdot Du)\xi^\gamma f'(u)  = \int (-j\cdot D\xi)\gamma\xi^{\gamma-1}f(u) \leq \gamma\int |j|\xi^{\gamma-1}|D\xi||f(u)|.\]
    Now, using \eqref{eq:mDpH_bound} and \eqref{eq:mDpHdotp_bound} and rearranging, we achieve 
    \begin{equation}\label{eq:test-estimate}
    \begin{aligned}
    \int |\xi Du|^{\gamma}f'(u) \leq & \enskip C\int |\xi Du|^{\gamma-1}|D\xi||f(u)| \\
    & \enskip + C\int |\xi|^{\gamma-1}|D\xi||f(u)| + C\int |\xi|^{\gamma}f'(u).
    \end{aligned}
    \end{equation}
    On the other hand, using Young's inequality, we get
    \[\int |\xi Du|^{\gamma-1}|D\xi||f(u)| \leq \sigma\int |\xi Du|^{\gamma}f'(u) + C_\sigma\int|D\xi|^{\gamma}\frac{|f(u)|^\gamma}{f'(u)^{\gamma-1}}\]
    for any $\sigma > 0$. We conclude by using this with a suitable value of $\sigma$ on the first term on the right-hand side of \eqref{eq:test-estimate} and using the analogous estimate on the middle term.
\end{proof}

\subsection{Standard Inequalities}\label{subsec:standard}

For completeness and to establish notation, we recall several classical inequalities that play a crucial role in our analysis. While these results are well-known, we present them with precise scaling properties essential for our proofs. The complete proofs of these inequalities can be found in standard references on elliptic PDE theory, such as \cite{evansPartialDifferentialEquations2010} and \cite{GilTru}.

Throughout what follows, $B_R$ denotes a ball of radius $R$ contained in $\Omega \subset \mathbb{R}^d$. When multiple balls $B_R$, $B_r$, etc.\ appear in the same context, they are assumed to be concentric. We use the notation $\tilde{B}_R$ for any ball that is not necessarily concentric with the others.

We use standard $L^p$-norm notation: for a measurable function $v$ and $0 < p < \infty$, we define
\[
    \|v\|_{L^p(B_R)} := \left(\int_{B_R} |v|^p\right)^{1/p}.
\]
As usual, $\|v\|_{L^\infty(B_R)}$ denotes the essential supremum of $|v|$ over $B_R$.
{
For $0<p<1$, these quantities are generally not norms; however, we use the norm symbol for notational simplicity. We also remark that these quantities can be infinite. }
For functions $v \geq 0$ almost everywhere, we extend this notation to include $-\infty \leq p < 0$ by setting
\[
    \|v\|_{L^p(B_R)} := \|v^{-1}\|^{-1}_{L^{-p}(B_R)}.
\]

Our estimates frequently involve the following scale-invariant norms:
\[
    R^{-d/p}\|v\|_{L^p(B_R)} \quad \text{and} \quad R^{1-d/p}\|Dv\|_{L^p(B_R)}.
\]
These quantities possess two important properties:
\begin{enumerate}
    \item They are non-decreasing in $p$ for fixed $v$, by Hölder's inequality.
    \item They are continuous in $p$, including at $p = \pm\infty$. For instance,
    \[
        \lim_{p \to \infty} R^{-d/p}\|v\|_{L^p(B_R)} = \|v\|_{L^\infty(B_R)}.
    \]
\end{enumerate}

We now state the fundamental functional inequalities used throughout this paper. Since the dimension $d$ remains fixed, we suppress the dependence of constants on~$d$.

\begin{lemma}[Sobolev inequality]\label{thm:sob-ineq}
    Let $v\in W^{1,1}(B_R)$. Let $1\leq p < d$ and let $p^* := pd/(d-p)$, so that $-d/p^* = 1-d/p$. Then
    \[R^{-d/p^*}\|v\|_{L^{p^*}(B_R)} \leq C_p\left(R^{-d/p}\|v\|_{L^{p}(B_R)} + R^{1-d/p}\|Dv\|_{L^p(B_R)}\right).\]
    Consequently, let $\Tilde{p}:= p(1+1/d) < p^*$, then
    \[R^{-d/\Tilde{p}}\|v\|_{L^{\Tilde{p}}(B_R)} \leq C_p\left(R^{-d/p}\|v\|_{L^{p}(B_R)} + R^{1-d/p}\|Dv\|_{L^p(B_R)}\right).\]
    Moreover, the last estimate holds for all $1\leq p\leq \infty$.
\end{lemma}

In the following, we denote the integral average of a measurable function~$v$ over a ball $B_R$ by $\avg_{B_R} v$; that is, 
\[\avg_{B_R} v = \frac{1}{|B_1|}R^{-d}\int_{B_R} v.\]

\begin{lemma}[Poincar\'e inequality]\label{thm:poincare}
    Let $v\in W^{1,\textcolor{brown}{p}}(B_R)$ and let $1\leq p \leq \infty$. Then,
    \[\|v-\avg_{B_R} v\|_{L^p(B_R)} \leq C_pR\|Dv\|_{L^p(B_R)}.\]
\end{lemma}

\begin{remark}\label{rem:morrey}
    When $p>d$, by Morrey's inequality, the left-hand side of the Sobolev inequality can be replaced with $\|v\|_{L^\infty(B_R)}$, and the left-hand side of the Poincar\'e inequality can be replaced with $\|v-\avg_{B_R} v\|_{L^\infty(B_R)}$.
\end{remark}

A key tool in our analysis is the following theorem of John and Nirenberg \cite{john-nirenberg}, which we state in a form adapted from Lemma 7 of \cite{serrin64}.
\begin{theorem}[John-Nirenberg]\label{thm:jn}
    Let $v\in L^1(B_R)$ 
    and suppose that
    \[r^{-d}\|v-\avg_{\Tilde{B}_r} v\|_{L^1(\Tilde{B}_r)}\]
    is uniformly bounded over all balls $\Tilde{B}_r\subset B_R$ (not necessarily concentric with~$B_R$). Then, there exists a sufficiently small $\epsilon > 0$ (depending on the aforementioned bound) such that
    \[ \|e^{\epsilon |v-\avg_{B_R} v| \,\,}\|_{L^1(B_R)} \leq \|2\|_{L^1(B_R)}.\]
\end{theorem}

Combining this result with the Poincar\'e inequality (Lemma~\ref{thm:poincare}), we obtain the following necessary corollary.

\begin{corollary}\label{cor:p-jn}
    Let $v\in W^{1,1}(B_R)$ and suppose that
    \[r^{1-d}\|Dv\|_{L^1(\Tilde{B}_r)}\]
    is uniformly bounded over all balls $\Tilde{B}_r\subset B_R$ (not necessarily concentric with~$B_R$). Then, there exists a sufficiently small $\epsilon > 0$ (depending on the aforementioned bound) such that
    \[R^{-d/\epsilon}\|e^{v}\|_{L^\epsilon(B_R)} \leq CR^{d/\epsilon}\|e^{v}\|_{L^{-\epsilon}(B_R)}.\]
\end{corollary}

\section{Hölder Continuity of the Value Function}
\label{sec:holder}

This section is dedicated to proving Theorem~\ref{thm:holder}. The main tool is the Caccioppoli-type Lemma~\ref{lem:basic-integral-estimate}.  
We first apply the lemma with a standard cutoff function $\xi$ and power-like functions $f$, obtaining bounds on the $W^{1,\gamma}$-norm of a power of~$u$ in a ball $B_r$ in terms of its $L^\gamma$-norm in a larger ball $B_{r'}$. These are then combined using a scale-invariant form of the Sobolev inequality, resulting in an estimate that bounds the $L^p$-norm with a higher exponent in a smaller ball by the $L^p$-norm with a lower exponent in a larger ball (see Section~\ref{subsec:sob-est}). 
Next, we iterate these estimates using Moser's method to derive a local boundedness lemma and a Harnack inequality (see Section~\ref{subsec:moser}). If $\gamma > d$, Moser's iteration can be avoided { due to Morrey's inequality, see Remark \ref{rem:gamma>d} and Lemma \ref{thm:loc-bound}.} Finally, in Section~\ref{subsec:conclusion}
we complete the proof of Theorem~\ref{thm:holder}.

\subsection{Sobolev Estimates}
\label{subsec:sob-est}

In this subsection, we use Lemma~\ref{lem:basic-integral-estimate} to derive scale-invariant $L^\gamma$-to-$W^{1,\gamma}$ estimates for power-like functions of \( u \). 

In our first lemma below, Lemma \ref{lem:sobest.q+}, our goal is to control the term \(\|D(|u|+R)^q \|_{L^\gamma(B_r)}\) by \(\|(|u|+R)^q \|_{L^\gamma(B_{r'})}\).
{
Directly controlling the gradient of a nonlinear expression like \((|u|+R)^q\) can be problematic because its direct use can lead to functions whose derivatives do not satisfy the boundedness conditions essential for our estimation techniques (such as Lemma \ref{lem:basic-integral-estimate}).
}
To overcome this, we use a truncated version of \((|u|+R)^q\), namely $F^q_{R,M}(|u|)$ which is discussed next.

The function \( F^q_{R,M} \colon [0,\infty) \to [0,\infty) \) is defined as follows:
\[
F^q_{R,M}(z) := \begin{cases}
    (z+R)^q \quad &\text{if }0\leq z\leq M,\\
    (M+R)^q + q(M+R)^{q-1}(z-M)\quad &\text{if }z\geq M,
\end{cases}
\]
for given positive values of \( q \), \( R \), and \( M \). This function is a shifted and asymptotically linearized version of the power function \( z^q \). 
Note that $M$ is an auxiliary parameter in which the estimate is uniform. Therefore, later, 
we can take \( M \) to infinity, noting that \( F^q_{R,M}(|u|) \) converges pointwise to \( (|u|+R)^q \). Introducing \( M \) ensures that \( F^q_{R,M} \) has a bounded derivative, making \( F^q_{R,M}(|u|) \in W^{1,\gamma}_{\text{loc}}(\Omega) \).
On the other hand, the shift represented by \( R \) is necessary, as the estimate would otherwise blow up when { we take the limit \( R \to 0^+ \).}

{
When applying Lemma~\ref{lem:sobest.q+} and Lemma~\ref{lem:sobest.q-} in the proof of Lemma~\ref{lem:after-sobolev}, the shift parameter (denoted $R$ in Lemmas~\ref{lem:sobest.q+} and~\ref{lem:sobest.q-}) is set to be the radius of the ball under consideration (also denoted $R$ in the definition of $a_{R,k}(\theta)$ in \eqref{eq:ark}), which justifies using the same symbol.}
 Using the length scale for the shift in $u$ is natural, since \( u \) itself scales with the domain size, as the Hamiltonian depends only on \( Du \).

In both of the Lemmas~\ref{lem:sobest.q+} and~\ref{lem:sobest.q-} below, given a pair of concentric balls $B_r\subset B_{r'}\subset \Omega$, we take the cutoff function $\xi$ to satisfy
\begin{equation}\label{eq:cutoff}
    0\leq \xi \leq 1, \quad \xi = 1 \text{ in } B_r, \quad \xi = 0 \text{ outside } B_{r'}, \quad |D\xi|\leq 2/(r'-r).
\end{equation}

\begin{lemma}\label{lem:sobest.q+}
Consider the setting of MFG System \ref{main}, suppose that Assumption~\ref{assume} holds.
Let $(m,u)$ be a solution in the sense of Definition~\ref{def:solution}, let $B_r \subset B_{r'} \subset \Omega$ be concentric balls, and let $q > (\gamma-1)/\gamma$, $R>0$, and $M>0$ be given. Then
    \begin{equation}
        \label{estimate}
        \|D(F^q_{R,M}(|u|))\|_{L^\gamma(B_r)} \leq C\left(\frac{1+q(\gamma(q-1)+1)^{-1}}{r'-r}+\frac{q}{R}\right)\|F^q_{R,M}(|u|)\|_{L^\gamma(B_{r'})},
    \end{equation}
    where $C$ depends only on the data in Assumption~\ref{assume}.
\end{lemma}
\begin{proof}
    We denote $Q := \gamma(q-1)+1 > 0$ and apply Lemma~\ref{lem:basic-integral-estimate} with $f\colon\mathbb{R}\to\mathbb{R}$ defined as
    \[f(u) := \begin{cases}
    -f(-u) \quad &\text{if }u\leq 0,\\
    Q^{-1}((u+R)^Q -R^Q) \quad &\text{if }0\leq u\leq M,\\
    Q^{-1}((M+R)^Q -R^Q) + (M+R)^{Q-1}(u-M)\quad &\text{if }u\geq M.
\end{cases}\]
Note that
\[f'(u) = (\min(|u|, M) + R)^{Q-1}\]
 is continuous, positive, and bounded. Therefore, the conditions in  Lemma~\ref{lem:basic-integral-estimate} are satisfied. For brevity, we define
 \[
 g(z) := q^{-1}F^q_{R,M}(z).
 \]
 Then, {taking into account that $Q-1=\gamma (q-1)$,} we observe that
\[f'(u) = g'(|u|)^\gamma.\] Moreover, 
\[\frac{(|f(u)|+Q^{-1}R^Q)^{\gamma}}{f'(u)^{\gamma-1}} = \begin{cases}
    Q^{-1}(|u|+R)^q \quad &\text{if }|u|\leq M,\\
    Q^{-1}(M+R)^q + (M+R)^{q-1}(|u|-M)\enskip &\text{if }|u|\geq M.
\end{cases}\]
Consequently,
\[\frac{|f(u)|^{\gamma}}{f'(u)^{\gamma-1}} \leq \max(q/Q, 1)^\gamma g(|u|)^\gamma \leq (1+q/Q)^\gamma g(|u|)^\gamma.\]
With these two observations, Lemma~\ref{lem:basic-integral-estimate} simplifies to
\[\int |\xi g'(|u|)Du|^{\gamma} \leq C\int (1+q/Q)^\gamma|D\xi g(|u|)|^{\gamma} + C\int |\xi g'(|u|)|^{\gamma},\]
which by \eqref{eq:cutoff} implies
\[\|Dg(|u|)\|_{L^\gamma(B_r)} \leq C\left(\frac{1+q/Q}{r'-r}\right)\|g(|u|)\|_{L^\gamma(B_{r'})} + C\|g'(|u|)\|_{L^\gamma(B_{r'})}.\]
Finally, we note 
\[g'(z)\leq \frac{q}{R}g(z),\]
and multiply both sides by $q$ to obtain \eqref{estimate}.
\end{proof}

The following lemma complements the previous one by providing a similar estimate in the range $q < (\gamma-1)/\gamma$. However, an extra hypothesis on the sign of $u$ is required.

\begin{lemma}\label{lem:sobest.q-}
Consider the setting of MFG System \ref{main}, suppose that Assumption~\ref{assume} holds.
Let $(m,u)$ be a solution in the sense of Definition~\ref{def:solution} and let $B_r \subset B_{r'} \subset \Omega$ be concentric balls, so that $u\geq 0$ on $B_{r'}$. Let $q < (\gamma-1)/\gamma$ and $R>0$ be given. Then
    \begin{equation}
        \|D(u+R)^q\|_{L^\gamma(B_r)} \leq C\left(\frac{|q||\gamma(q-1)+1|^{-1}}{r'-r}+\frac{|q|}{R}\right)\|(u+R)^q\|_{L^{\gamma}(B_{r'})}.
    \end{equation}
    Moreover,
    \begin{equation}
        \|D\log(u+R)\|_{L^\gamma(B_r)} \leq C\left(\frac{1}{r'-r}+\frac{1}{R}\right)\|1\|_{L^{\gamma}(B_{r'})}.
    \end{equation}
    Recall that $C$ depends only on the data in Assumption~\ref{assume}.
\end{lemma}

\begin{proof}
We denote $Q := \gamma(q-1) + 1 < 0$ and
apply Lemma~\ref{lem:basic-integral-estimate} with the function $f\colon [0,\infty)\to \mathbb{R}$ defined as
\[f(u) := Q^{-1}(u+R)^Q.\]
Because Lemma~\ref{lem:basic-integral-estimate} requires a function 
defined in $\mathbb{R}$, we can extend $f$ suitably 
to the negative real numbers.  However, this extension is irrelevant as we have $u\geq 0$ by hypothesis.
Note that
\[f'(u) = (u+R)^{Q-1},\]
which is continuous, positive, and bounded for $u\geq 0$. 
Accordingly, Lemma~\ref{lem:basic-integral-estimate} yields
\begin{equation}\label{eq:powertest.Q-}
\int |\xi (u+R)^{q-1}Du|^{\gamma} \leq C\int |Q|^{-\gamma}|D\xi(u+R)^q|^{\gamma} + C\int |\xi (u+R)^{q-1}|^{\gamma}.
\end{equation}
Therefore, using \eqref{eq:cutoff} we obtain
\begin{equation}
\label{estq}
\|(u+R)^{q-1}Du\|_{L^\gamma(B_r)} \leq C\left(\frac{|Q|^{-1}}{r'-r}\right)\|(u+R)^q\|_{L^\gamma(B_{r'})} + C\|(u+R)^{q-1}\|_{L^\gamma(B_{r'})}.
\end{equation}
Multiplying both sides by $|q|$ and noting that
\[(u+R)^{q-1}\leq \frac{1}{R}(u+R)^q,\]
we obtain the first estimate. Moreover, the special case $q=0$ in \eqref{estq} immediately gives the second estimate.
\end{proof}

The next lemma provides a reverse Hölder inequality for~$u$, which arises from combining the Sobolev inequality of Lemma~\ref{thm:sob-ineq} with Lemmas~\ref{lem:sobest.q+} and~\ref{lem:sobest.q-}. To facilitate the formulation of the estimate, we define
\begin{equation}\label{eq:ark}
    a_{R,k}(\theta) := R^{-d/\theta}\||u|+R \|_{L^{\theta}(B_{R(1+k/|\theta|)})},
\end{equation}
where $B_R\subset\Omega$ is a ball, $k$ is a positive value, and $|\theta|$ is large enough so that $B_{R(1+k/|\theta|)}\subset\Omega$. 
In this context, all the balls $B_{R(1+k/|\theta|)}$ are concentric.

\begin{lemma}\label{lem:after-sobolev}
Consider the setting of MFG System \ref{main}, suppose that Assumption~\ref{assume} holds.
Let $(m,u)$ be a solution in the sense of Definition~\ref{def:solution}, let $B_{2R}\subset\Omega$ be a ball, and let $k>0$ be given. Consider $a_{R,k}$ defined by \eqref{eq:ark}. Then,
\begin{enumerate}
    \item We have
    \[a_{R,k}(\theta(1+1/d)) \leq \left(C_{k}\theta^2\right)^{\gamma/\theta}a_{R,k}(\theta), \qquad \text{for } \theta \geq \gamma-1+k.\]
    \item Moreover, if $u\geq 0$ in $B_{2R}$, we have
    \begin{align}
    & a_{R,k}(\theta(1+1/d)) \leq \left(C_{k}\theta^2\right)^{\gamma/\theta}a_{R,k}(\theta), \qquad \text{for } k\leq \theta\leq \gamma-1-k,\\
    & a_{R,k}(\theta(1+1/d)) \geq \left(C_{k}\theta^2\right)^{\gamma/\theta}a_{R,k}(\theta), \qquad \text{for } \theta\leq -k .
\end{align}
\end{enumerate}
Note that $C_{k}$ is independent of the solution $(m,u)$, the scale $R$, and the exponent $\theta$.
\end{lemma}
\begin{proof}
    We prove the first inequality in detail but only sketch the proof for the other inequalities since they are entirely analogous. First, we apply Lemma~\ref{lem:sobest.q+} with $q := \theta/\gamma$, with $R$ as given (i.e.\ the shift $R$ of Lemma~\ref{lem:sobest.q+} is to be set equal to the radius $R$ from here), with $M$ arbitrary, and with
    \[r' := R\left(1+\frac{k}{|\theta|}\right) \qquad\text{and}\qquad r := R\left(1+\frac{k}{(1+1/d)|\theta|}\right).\]
    Consequently, since $\theta>\gamma-1+k$, we have
    \[1+q(\gamma(q-1)+1)^{-1} \leq \frac{C}{k}\theta, \qquad \frac{1}{r'-r}\leq \frac{C}{kR}\theta, \qquad \frac{q}{R}\leq \frac{C}{kR}\theta^2.\]
    Thus, we get
    \[\|D(F^q_{R,M}(|u|))\|_{L^\gamma(B_r)} \leq C_{k}\frac{\theta^2}{R}\|F^q_{R,M}(|u|)\|_{L^\gamma(B_{r'})}.\]
    Equivalently, we have
    \begin{equation}
        R^{1-d/\gamma}\|D(F^q_{R,M}(|u|))\|_{L^\gamma(B_r)} \leq C_{k}\theta^2\left(R^{-d/\gamma}\|F^q_{R,M}(|u|)\|_{L^\gamma(B_{r'})}\right).
    \end{equation}
    Next, we use Sobolev inequality, Lemma \ref{thm:sob-ineq}, with $v := F^q_{R,M}(|u|)$ and $p := \gamma$ on the ball $B_r$. 
    Because $\theta>k$, we can assume that $C_k \theta^2>1$ by making the constant larger if necessary. Then, we have 
    \[
    \|v\|_{L^\gamma(B_r)} \leq \|v\|_{L^\gamma(B_{r'})} \leq C_{k}\theta^2\|v\|_{L^\gamma(B_{r'})}.
    \]
    Taking into account that $R\leq r<r'\leq 2R$, 
Lemma \ref{thm:sob-ineq} yields
\[R^{-d/\Tilde{\gamma}}\|F^q_{R,M}(|u|)\|_{L^{\Tilde{\gamma}}(B_r)} \leq C_{k}\theta^2\left(R^{-d/\gamma}\|F^q_{R,M}(|u|)\|_{L^\gamma(B_{r'})}\right)\]
for $\Tilde{\gamma} := \gamma(1+1/d)$, taking into account that 
$\Tilde \gamma<\gamma^*$.

Now, we pass to the limit $M\to\infty$ and use the monotone convergence theorem to obtain
    \begin{equation}\label{eq:after-sobolev-before-power}
        R^{-d/\Tilde{\gamma}}\|(|u|+R)^q\|_{L^{\Tilde{\gamma}}(B_r)} \leq C_{k}\theta^2\left(R^{-d/\gamma}\|(|u|+R)^q\|_{L^\gamma(B_{r'})}\right).
    \end{equation}
    Now, we get the first inequality by raising both sides to the power of $1/q = \gamma/\theta$.

    As for the other inequalities, we use Lemma~\ref{lem:sobest.q-} instead of Lemma~\ref{lem:sobest.q+} with the same choice of $q$, $R$, $r'$, and $r$ (we do not need any arbitrary $M$ this time). We can use Lemma~\ref{lem:sobest.q-} since $r'\leq 2R$, hence $u\geq 0$ in $B_{r'}$. Then, the estimates above proceed exactly in the same way with $F^{q}_{R,M}(|u|)$ replaced by $(u+R)^q$. Thus, we obtain \eqref{eq:after-sobolev-before-power} without needing the monotone convergence theorem. In the final step of raising to the power of $1/q$, the inequality remains the same if $\theta>0$ but is reversed if $\theta<0$. 
\end{proof}
\begin{remark}\label{rem:gamma>d}
    Suppose $\gamma > d$. Then, the left-hand sides in Lemma~\ref{lem:after-sobolev} can be replaced with $a_{R,k}(\infty) = \||u|+R\|_{L^\infty(B_R)}$ in the case $\theta > 0$, and with $a_{R,k}(-\infty) = \|u+R\|_{L^{-\infty}(B_R)}$ in the case $\theta < 0$. This follows from Remark~\ref{rem:morrey}.
\end{remark}

\subsection{Moser's Iteration Method}
\label{subsec:moser}

{
To upgrade our $L^\gamma$–bounds on $u$ to uniform control and ultimately prove Harnack‐type inequalities, we employ the classical Moser iteration scheme.  
For this, we iterate Lemma~\ref{lem:after-sobolev}. Namely, we substitute values from suitable geometric sequences in the place of $\theta$ and combine the resulting estimates. Consequently, we achieve a $L^\gamma$-to-$L^\infty$ estimate for $u$, proving that $u$ is (locally) bounded, and we also obtain a Harnack-type estimate for~$u$.}

\begin{lemma}\label{thm:loc-bound}
Consider the setting of MFG System \ref{main}, suppose that Assumption~\ref{assume} holds.
Let $(m,u)$ be a solution in the sense of Definition \ref{def:solution}, let $B_{2R}\subset \Omega$ be a ball, and let $\lambda > \gamma-1$ be given. Then
    \begin{equation}
    \label{supest}
        \|u\|_{L^\infty(B_R)} \leq C_{\lambda} (R^{-d/\lambda}\|u\|_{L^{\lambda}(B_{2R})} + R).
    \end{equation}
Note that the constant $C_\lambda$ depends only on $\lambda$ and the data in Assumption~\ref{assume}, so it is independent of $(m,u)$ and $R$.
\end{lemma}
\begin{proof}
For integers $j\geq 0$, we denote $\theta_j := \lambda(1+1/d)^j$ and apply Lemma~\ref{lem:after-sobolev} with $\theta := \theta_j$ and $k := \theta_0 - \gamma + 1 = \lambda - \gamma + 1$. Consequently we get
\begin{equation}\label{eq:moser-iteratable}
    a_{R,k}(\theta_{j+1}) \leq (C_\lambda^{j+1})^{(1+1/d)^{-j}}a_{R,k}(\theta_j)
\end{equation}
where $C_\lambda>1$ is chosen as any constant greater than $(C_k\lambda^2)^{\gamma/\lambda}$ and $(1+1/d)^{2\gamma/\lambda}$.
Next, chaining the inequalities \eqref{eq:moser-iteratable} as $j$ ranges through $0$ to $(n-1)$, we obtain
\begin{equation}\label{eq:moser-iterated}
    a_{R,k}(\theta_n) \leq C_\lambda^{\sum_{j=0}^\infty (j+1)(1+1/d)^{-j}} a_{R,k}(\lambda) =: \Tilde C_\lambda a_{R,k}(\lambda).
\end{equation}
Now, note that
\[a_{R,k}(\lambda) \leq R^{-d/\lambda}\|(|u|+R)\|_{L^{\lambda}(B_{2R})} \leq R^{-d/\lambda}\|u\|_{L^{\lambda}(B_{2R})} + R,\]
and
\[\lim_{n\to\infty} a_{R,k}(\theta_n) = a_{R,k}(\infty) = \|u\|_{L^\infty(B_R)} + R > \|u\|_{L^\infty(B_R)}.\]
Therefore, passing to the limit in \eqref{eq:moser-iterated}, we obtain \eqref{supest}.

Finally, note that if $\gamma > d$, the proof can be immediately achieved without iteration due to Remark~\ref{rem:gamma>d}.
\end{proof}

In the remainder of this section, we denote the essential supremum (infimum) of a measurable function $v$ over a ball $B_R$ by
\[\max_{B_R} v\enskip (\min_{B_R} v).\]
Note that in the case $v\geq 0$, we have
\[\max_{B_R} v = \|v\|_{L^{\infty}(B_R)} \qquad\text{and}\qquad \min_{B_R} v = \|v\|_{L^{-\infty}(B_R)}.\]
\begin{lemma}[Harnack inequality]\label{thm:harnack}
Consider the setting of MFG System \ref{main}, suppose that Assumption~\ref{assume} holds.
Let $(m,u)$ be a solution in the sense of Definition \ref{def:solution} and let $B_{2R}\subset \Omega$ be a ball, so that $u\geq 0$ in $B_{2R}$. Then
    \[\max_{B_R} u \leq C(\min_{B_R} u + R),\]
    where $C$ depends only on the data in Assumption~\ref{assume}.
\end{lemma}

\begin{proof}
We first apply the second inequality in Lemma~\ref{lem:sobest.q-} with $R$ as given here, $B_r$ replaced by any ball $\Tilde{B}_r\subset B_{3R/2}$ (not necessarily concentric), and $r' = 4r/3$. We can apply the lemma since $\Tilde{B}_{4r/3}\subset B_{2R}$, hence $u\geq 0$ in $\Tilde{B}_{4r/3}$. Consequently, since $R\geq 2r/3$, we get
\[\|D\log(u+R)\|_{L^\gamma(\Tilde{B}_{r})} \leq \frac{C}{r}\|1\|_{L^{\gamma}(\Tilde{B}_{4r/3})},\]
or, equivalently,
\[r^{1-d/\gamma}\|D\log(u+R)\|_{L^\gamma(\Tilde{B}_{r})} \leq C.\]
Therefore, Corollary~\ref{cor:p-jn} implies that there exists a sufficiently small $\epsilon > 0$, which depends only on the data in Assumption~\ref{assume}, such that
\begin{equation}\label{eq:jn-link}
    R^{-d/\epsilon}\|u+R\|_{L^\epsilon(B_{3R/2})} \leq CR^{d/\epsilon}\|u+R\|_{L^{-\epsilon}(B_{3R/2})}.
\end{equation}
Thus, to complete the proof, it suffices to prove
\begin{equation}\label{eq:har-complete}
\begin{alignedat}{2}
   & \|u+R\|_{L^{\infty}(B_R)} &&\leq CR^{-d/\epsilon}\|u+R\|_{L^{\epsilon}(B_{3R/2})}, \\
   & \|u+R\|_{L^{-\infty}(B_R)} \, &&\geq C^{-1}R^{d/\epsilon}\|u+R\|_{L^{-\epsilon}(B_{3R/2})},
\end{alignedat}
\end{equation}
which we do separately below. We also note that if $\gamma>d$, the use of Corollary~\ref{cor:p-jn} above can be improved due to the improvement of Lemma~\ref{thm:poincare} given in Remark~\ref{rem:morrey}, resulting in the immediate conclusion of the proof instead of \eqref{eq:jn-link}. Thus, the proof of \eqref{eq:har-complete} below is only needed for $\gamma\leq d$, although we do not assume this in the proof.
\end{proof}

\begin{proof}[Proof of \eqref{eq:har-complete}]
    We prove the first inequality in detail but only sketch the proof for the second inequality since it is entirely analogous. The proof is based on iterating Lemma~\ref{lem:after-sobolev} with Moser's method, as in the proof of Lemma~\ref{thm:loc-bound} above. First, we choose a sufficiently large integer $N\geq 0$ such that
    \[\Tilde{\epsilon} := (\gamma-1)(1+1/2d)^{-1}(1+1/d)^{-N} \leq \epsilon.\]
    Clearly, $N$ and $\Tilde{\epsilon}$ depend only on the data in Assumption~\ref{assume}. Next we denote $\theta_j := \Tilde{\epsilon}(1+1/d)^j$ and apply Lemma~\ref{lem:after-sobolev} with $\theta := \theta_j$ and 
    \[k := \min\left(\frac{\Tilde{\epsilon}}{2}, \frac{\gamma-1}{2d+1}\right).\]
   We can apply the lemma because $u\geq 0$ in $B_{2R}$ and
    \[\begin{aligned}
        k \leq \ & \theta_j\leq \gamma-1-k \quad && \text{for } 0\leq j\leq N,\\
         & \theta_j \geq \gamma-1+k && \text{for } j > N.
    \end{aligned}\]
    Consequently, we get
    \begin{equation}\label{eq:moser-iteratable+}
    a_{R,k}(\theta_{j+1}) \leq (C^{j+1})^{(1+1/d)^{-j}}a_{R,k}(\theta_j),
\end{equation}
where $C$ is chosen as any constant greater than $(C_k\Tilde{\epsilon}^2)^{\gamma/\Tilde{\epsilon}}$ and $(1+1/d)^{2\gamma/\Tilde{\epsilon}}$.
Next, chaining the inequalities \eqref{eq:moser-iteratable+} as $j$ ranges through $0$ to $(n-1)$, we obtain
\begin{equation}\label{eq:moser-iterated+}
    a_{R,k}(\theta_n) \leq C^{\left(\sum (j+1)(1+1/d)^{-j}\right)} a_{R,k}(\Tilde{\epsilon}) =: C a_{R,k}(\Tilde{\epsilon}).
\end{equation}
Now note that
\[a_{R,k}(\Tilde{\epsilon}) \leq R^{-d/\Tilde{\epsilon}}\|u+R\|_{L^{\Tilde{\epsilon}}(B_{3R/2})} \leq R^{-d/{\epsilon}}\|u+R\|_{L^{{\epsilon}}(B_{3R/2})},\]
and
\[\lim_{n\to\infty} a_{R,k}(\theta_n) = a_{R,k}(\infty) = \|u+R\|_{L^\infty(B_R)}.\]
Therefore, passing to the limit in \eqref{eq:moser-iterated+}, we conclude the first inequality in~\eqref{eq:har-complete}. For the second inequality, we apply Lemma~\ref{lem:after-sobolev} with $\theta := -\theta_j$ and the same~$k$, and follow exactly the same steps.
\end{proof}

\subsection{Proof of the Main Theorem}\label{subsec:conclusion}

In this subsection, we prove Theorem \ref{thm:holder} as an immediate corollary of a result we establish next, namely the Hölder seminorm estimate in Lemma~\ref{thm:holder+} (see Corollary~\ref{cor:holder}). Obtaining such estimates from Harnack-type inequalities is standard in the theory of elliptic PDE. Here, we use the Harnack inequality in Lemma~\ref{thm:harnack}, thus obtaining an analogous result
for MFGs. 
For $\gamma > d$, the Hölder continuity follows from Morrey's inequality. 
However, if $\gamma\leq d$, it substantially improves the regularity of solutions.  However, in the proof, we do not assume anything about $\gamma$.

To prove Lemma~\ref{thm:holder+}, we first obtain a quantitative estimate of oscillation decay in Lemma~\ref{lem:osc}. In the sequel, we denote the essential oscillation of a measurable function $v$ over a ball $B_R$ by $\osc_{B_R} v$, that is,
\[\osc_{B_R} v := \max_{B_R} v - \min_{B_R} v.\]

\begin{lemma}\label{lem:osc}
Consider the setting of MFG System \ref{main}, suppose that Assumption~\ref{assume} holds. Let $(m,u)$ be a solution in the sense of Definition~\ref{def:solution} and consider a ball $B_{2R}\subset\Omega$. Then, there exists $\mu>0$ depending only on the data in Assumption~\ref{assume},
    \[\osc_{B_R} u \leq 2^{-\mu}\osc_{B_{2R}} u + 2R.\]
\end{lemma}
\begin{proof}
First note that Lemma~\ref{thm:loc-bound} with $\lambda:=\gamma$ implies that $u$ is locally bounded. Hence, $\min_{B_{2R}} u$ and $\max_{B_{2R}} u$ are finite. Secondly, when $(m,u)$ is a solution to MFG System~\ref{main} with given Hamiltonian~$H$ and $c$ is an arbitrary constant, then $(m,u+c)$ is also a solution with the same Hamiltonian and $(m,c-u)$ is a solution with the new Hamiltonian
\[H^{-}(x,p,m) := H(x,-p,m),\]
which also satisfies Assumption~\ref{assume} with the same parameters. Thus, we apply Lemma~\ref{thm:harnack} to the solutions
\[(m, u-\min_{B_{2R}}u) \qquad\text{and}\qquad (m, \max_{B_{2R}}u - u).\]
Accordingly, we obtain
\begin{align}
    & \max_{B_{R}}u - \min_{B_{2R}}u \leq  C(\min_{B_{R}}u - \min_{B_{2R}}u + R), && \\
    & \max_{B_{2R}}u - \min_{B_{R}}u \leq  C(\max_{B_{2R}}u - \max_{B_{R}}u + R). &&
\end{align}
Adding these two inequalities and rearranging, we get 
\[\osc_{B_R} u \leq \left(\frac{C-1}{C+1}\right) \osc_{B_{2R}} u + \frac{2C}{C+1} R.
\]
Now, the result follows with $2^{-\mu} =\frac{C-1}{C+1}$ and by noting $\frac{2C}{C+1}\leq 2$.
\end{proof}

We now apply Lemma~\ref{lem:osc} to 
prove the Hölder seminorm estimate in Lemma~\ref{thm:holder+}.

\begin{lemma}\label{thm:holder+}
Consider the setting of MFG System~\ref{main}, suppose that Assumption~\ref{assume} holds. Let $(m,u)$ be a solution in the sense of Definition~\ref{def:solution} and consider a ball $B_{4R}\subset\Omega$. Then, for a certain sufficiently small $\mu>0$ depending only on the data in Assumption~\ref{assume},
\begin{equation}
\label{holder}
    R^{\mu} [u]_{C^{0,\mu}(B_R)} \leq C(R^{-d/\gamma}\|u\|_{L^\gamma(B_{4R})} + R).
\end{equation}
\end{lemma}


\begin{proof}
Let $\mu > 0$ be as in Lemma \ref{lem:osc} and, 
without loss of generality (decreasing it if necessary), we suppose that $\mu < 1$. To prove the Hölder seminorm estimate, we establish the following inequality
\begin{equation}\label{eq:holder}
    R^\mu \frac{\osc_{\Tilde{B}_r} u}{r^\mu} \leq C(\|u\|_{L^\infty(B_{2R})} + R)
\end{equation}
for all balls $\Tilde{B}_r$ with $r\leq R$ centered at a point in $B_R$. Then \eqref{holder} follows from Lemma~\ref{thm:loc-bound} with $\lambda:=\gamma$. 

We fix such a ball $\Tilde{B}_r$, denote $R_j := 2^{-j} R$ for integers $j\geq 0$, and apply Lemma~\ref{lem:osc} for the concentric ball pairs $(\Tilde{B}_{R_j}, \Tilde{B}_{R_{j-1}})$ to obtain
\begin{equation}\label{eq:osc-decay-j1}
    \osc_{\Tilde{B}_{R_j}} u \leq 2^{-\mu} \osc_{\Tilde{B}_{R_{j-1}}} u + 2R_j,
\end{equation}
for $j\geq 1$. Next, we denote
\[a_j := 2^{j\mu}\osc_{\Tilde{B}_{R_j}} u,\]
to simplify \eqref{eq:osc-decay-j1} to
\begin{equation}\label{eq:osc-decay-j2}
    a_j \leq a_{j-1} + 2R (2^{1-\mu})^{-j}.
\end{equation}
Then, let $n\geq 0$ be the greatest integer satisfying $r\leq R_n$. Chaining the inequalities \eqref{eq:osc-decay-j2} as $j$ ranges through $1$ to $n$, we get
\begin{equation}
    a_n \leq a_0 + 2R\cdot \sum (2^{1-\mu})^{-j} \leq a_0 + CR \leq C (\|u\|_{L^\infty(B_{2R})} + R),
\end{equation}
where the last inequality holds since $\Tilde{B}_R\subset B_{2R}$. Now, since $R_{n+1} < r\leq R_n$, we have
\[
R^\mu \frac{\osc_{\Tilde{B}_r} u}{r^\mu} \leq R^\mu \frac{\osc_{\Tilde{B}_{R_n}} u}{{R_{n+1}}^\mu} = 4^\mu a_n.
\]
Hence, we conclude \eqref{eq:holder}.
\end{proof}

The previous result yields the Hölder continuity of the solution. 

\begin{corollary}\label{cor:holder}
    Consider the setting of MFG System \ref{main}, suppose that Assumption~\ref{assume} holds. Let $(m,u)$ be a solution in the sense of Definition~\ref{def:solution}. Then, for a certain sufficiently small $\mu>0$ depending only on the data in Assumption~\ref{assume}, we have $u\in C^{0,\mu}_{\text{loc}}$.
\end{corollary}

%% file: main.bbl
\begin{thebibliography}{10}

\bibitem{AcCiMa}
Y.~Achdou, M.~Cirant, and M.~Bardi.
\newblock Mean field games models of segregation.
\newblock {\em Mathematical Models \& Methods in Applied Sciences},
  27(1):75--113, 2017.

\bibitem{alharbi2023first}
A.~M. Alharbi, Y.~Ashrafyan, and D.~Gomes.
\newblock A first-order mean-field game on a bounded domain with mixed boundary
  conditions.
\newblock {\em arXiv preprint arXiv:2305.15952}, 2023.

\bibitem{almulla2017two}
N.~Almulla, R.~Ferreira, and D.~Gomes.
\newblock Two numerical approaches to stationary mean-field games.
\newblock {\em Dynamic Games and Applications}, 7(4):657--682, 2017.

\bibitem{Bagagiolo2022}
F.~Bagagiolo, S.~Faggian, R.~Maggistro, and R.~Pesenti.
\newblock Optimal control of the mean field equilibrium for a pedestrian
  {{Tourists}}'{{Flow}} model.
\newblock {\em Networks and Spatial Economics}, 2019.

\bibitem{Bauso20143475}
D.~Bauso and R.~Pesenti.
\newblock Opinion dynamics, stubbornness and mean-field games.
\newblock {\em Proceedings of the IEEE Conference on Decision and Control},
  2015-February(February):3475--3480, 2014.

\bibitem{Bolouki2014556}
S.~Bolouki, M.~Siami, R.~Malhame, and N.~Motee.
\newblock Eminence grise coalitions in opinion dynamics.
\newblock {\em 2014 52nd Annual Allerton Conference on Communication, Control,
  and Computing, Allerton 2014}, pages 556--562, 2014.

\bibitem{Caines2017}
P.~E. Caines, M.~Huang, and R.~P. Malham{\'e}.
\newblock Mean {{Field Games}}.
\newblock In T.~Basar and G.~Zaccour, editors, {\em Handbook of {{Dynamic Game
  Theory}}}, pages 1--28. Springer International Publishing, Cham, 2017.

\bibitem{Cd1}
P.~Cardaliaguet.
\newblock Long {{Time Average}} of {{First Order Mean Field Games}} and {{Weak
  KAM Theory}}.
\newblock {\em Dynamic Games and Applications}, 3(4):473--488, Dec. 2013.

\bibitem{Card1order}
P.~Cardaliaguet.
\newblock Weak {{Solutions}} for {{First Order Mean Field Games}} with {{Local
  Coupling}}.
\newblock In P.~Bettiol, P.~Cannarsa, G.~Colombo, M.~Motta, and F.~Rampazzo,
  editors, {\em Analysis and {{Geometry}} in {{Control Theory}} and Its
  {{Applications}}}, pages 111--158. Springer International Publishing, Cham,
  2015.

\bibitem{cardaliaguet2018short}
P.~Cardaliaguet.
\newblock A short course on {{Mean}} field games.
\newblock 2018.

\bibitem{cardaliaguetMeanFieldGames2014}
P.~Cardaliaguet and P.~Graber.
\newblock Mean field games systems of first order.
\newblock {\em ESAIM: Control, Optimisation and Calculus of Variations}, 21,
  Jan. 2014.

\bibitem{cardaliaguetMeanFieldGame2018}
P.~Cardaliaguet and C.-A. Lehalle.
\newblock Mean field game of controls and an application to trade crowding.
\newblock {\em Mathematics and Financial Economics}, 12(3):335--363,
  2018/06/01, 2018.

\bibitem{San16}
P.~Cardaliaguet, A.~M{\'e}sz{\'a}ros, and F.~Santambrogio.
\newblock First order mean field games with density constraints: Pressure
  equals price.
\newblock {\em Siam Journal On Control and Optimization}, 54(5):2672--2709,
  2016.

\bibitem{cirant}
M.~Cirant.
\newblock Multi-population {{Mean Field Games}} systems with {{Neumann}}
  boundary conditions.
\newblock 103(5):1294--1315, 2015.

\bibitem{2016arXiv161107187C}
M.~Cirant, D.~A. Gomes, E.~A. Pimentel, and H.~{S{\'a}nchez-Morgado}.
\newblock On some singular mean-field games.
\newblock {\em Journal of Dynamics and Games}, 8(4):445--465, Tue Mar 02
  00:00:00 UTC 2021.

\bibitem{20.500.11769_492793}
G.~Di~Fazio.
\newblock Regularity for {{Elliptic Equations}} under {{Minimal Assumptions}}.
\newblock {\em Journal of the Indonesian Mathematical Society}, pages 101--127,
  Mar. 2020.

\bibitem{20.500.11769_585952}
G.~Di~Fazio, M.~S. Fanciullo, D.~D. Monticelli, S.~Rodney, and P.~Zamboni.
\newblock Matrix weights and regularity for degenerate elliptic equations.
\newblock {\em Nonlinear Analysis}, 237:113363, Dec. 2023.

\bibitem{20.500.11769_518677}
G.~Di~Fazio, M.~S. Fanciullo, and P.~Zamboni.
\newblock Boundary {{Harnack Type Inequality}} and {{Regularity}} for
  {{Quasilinear Degenerate Elliptic Equations}}.
\newblock In V.~Vespri, U.~Gianazza, D.~D. Monticelli, F.~Punzo, and
  D.~Andreucci, editors, {\em Harnack {{Inequalities}} and {{Nonlinear
  Operators}}}, pages 139--157, Cham, 2021. Springer International Publishing.

\bibitem{20.500.11769_633169}
G.~Di~Fazio, M.~S. Fanciullo, and P.~Zamboni.
\newblock Smoothness for {{Degenerate Elliptic Equations}} with {{Matrix
  Weights}}.
\newblock In A.~Ashyralyev, M.~Ruzhansky, and M.~A. Sadybekov, editors, {\em
  Analysis and {{Applied Mathematics}}}, pages 163--169, Cham, 2024. Springer
  Nature Switzerland.

\bibitem{evans2003some}
L.~C. Evans.
\newblock Some new {{PDE}} methods for weak {{KAM}} theory.
\newblock {\em Calculus of Variations and Partial Differential Equations},
  17(2):159--177, 2003.

\bibitem{E2}
L.~C. Evans.
\newblock Further {{PDE}} methods for weak {{KAM}} theory.
\newblock {\em Calculus of Variations and Partial Differential Equations},
  35(4):435--462, 2009.

\bibitem{evansPartialDifferentialEquations2010}
L.~C. Evans.
\newblock {\em Partial Differential Equations}, volume~19 of {\em Graduate
  {{Studies}} in {{Mathematics}}}.
\newblock American Mathematical Society, Providence, RI, second edition, 2010.

\bibitem{FTT20}
O.~F{\'e}ron, P.~Tankov, and L.~Tinsi.
\newblock Price formation and optimal trading in intraday electricity markets
  with a major player.
\newblock {\em Risks}, 8(4):1--1, Dec. 2020.

\bibitem{FG2}
R.~Ferreira and D.~Gomes.
\newblock Existence of weak solutions to stationary mean-field games through
  variational inequalities.
\newblock {\em Siam Journal On Mathematical Analysis}, 50(6):5969--6006, 2018.

\bibitem{FGT1}
R.~Ferreira, D.~Gomes, and T.~Tada.
\newblock Existence of weak solutions to first-order stationary mean-field
  games with {{Dirichlet}} conditions.
\newblock {\em Proceedings of the American Mathematical Society},
  147(11):4713--4731, 2019.

\bibitem{FeGoTa21}
R.~Ferreira, D.~Gomes, and T.~Tada.
\newblock Existence of weak solutions to time-dependent mean-field games.
\newblock {\em Nonlinear Analysis}, 212:Paper No. 112470, 31, 2021.

\bibitem{FT20}
M.~Fujii and A.~Takahashi.
\newblock A mean field game approach to equilibrium pricing with market
  clearing condition.
\newblock Papers 2003.03035, arXiv.org, Mar. 2020.

\bibitem{GLasryLionsMoll}
X.~Gabaix, J.-M. Lasry, P.-L. Lions, and B.~Moll.
\newblock The dynamics of inequality.
\newblock {\em Econometrica : journal of the Econometric Society},
  84(6):2071--2111, 2016.

\bibitem{GilTru}
D.~Gilbarg and N.~S. Trudinger.
\newblock {\em Elliptic Partial Differential Equations of Second Order}.
\newblock Classics in Mathematics. Springer-Verlag, Berlin, 2001.

\bibitem{20.500.11769_489964}
D.~F. Giuseppe, F.~M. Stella, and Z.~Pietro.
\newblock Regularity up to the boundary for some degenerate elliptic operators.
\newblock {\em Applicable Analysis}, 101(10):3563--3575, July 2022.

\bibitem{GomesMS19}
D.~Gomes, D.~Marcon, and F.~Al~Saleh.
\newblock The current method for stationary mean-field games on networks.
\newblock In {\em 58th {{IEEE}} Conference on Decision and Control, {{CDC}}
  2019, Nice, France, December 11-13, 2019}, pages 305--310. IEEE, 2019.

\bibitem{MR4175148}
D.~Gomes, H.~Mitake, and K.~Terai.
\newblock The selection problem for some first-order stationary mean-field
  games.
\newblock {\em Networks and Heterogeneous Media}, 15(4):681--710, 2020.

\bibitem{GMS}
D.~Gomes, J.~Mohr, and R.~R. Souza.
\newblock Discrete time, finite state space mean field games.
\newblock {\em Journal de Math{\'e}matiques Pures et Appliqu{\'e}es},
  93(2):308--328, 2010.

\bibitem{GMS2}
D.~Gomes, J.~Mohr, and R.~R. Souza.
\newblock Continuous time finite state mean-field games.
\newblock {\em Appl. Math. and Opt.}, 68(1):99--143, 2013.

\bibitem{GNPr216}
D.~Gomes, L.~Nurbekyan, and M.~Prazeres.
\newblock Explicit solutions of one-dimensional, first-order, stationary
  mean-field games with congestion.
\newblock {\em 2016 IEEE 55th Conference on Decision and Control, CDC 2016},
  pages 4534--4539, 2016.

\bibitem{Gomes2016b}
D.~Gomes, L.~Nurbekyan, and M.~Prazeres.
\newblock One-dimensional stationary mean-field games with local coupling.
\newblock {\em Dyn. Games and Applications}, 2017.

\bibitem{GPV}
D.~Gomes, E.~Pimentel, and V.~Voskanyan.
\newblock {\em Regularity Theory for Mean-Field Game Systems}.
\newblock {{SpringerBriefs}} in Mathematics. Springer, [Cham], 2016.

\bibitem{GM}
D.~Gomes and H.~S{\'a}nchez~Morgado.
\newblock A stochastic {{Evans-Aronsson}} problem.
\newblock {\em Transactions of the American Mathematical Society},
  366(2):903--929, 2014.

\bibitem{GS}
D.~Gomes and J.~Sa{\'u}de.
\newblock Mean field games models---a brief survey.
\newblock {\em Dynamic Games and Applications}, 4(2):110--154, 2014.

\bibitem{MR4215224}
D.~Gomes and J.~Sa{\'u}de.
\newblock A {{Mean-Field Game Approach}} to {{Price Formation}}.
\newblock {\em Dynamic Games and Applications}, 11(1):29--53, 2021.

\bibitem{graber2017sobolev}
P.~J. Graber and A.~R. M{\'e}sz{\'a}ros.
\newblock Sobolev regularity for first order mean field games.
\newblock {\em Ann. Inst. H. Poincar{\'e} Anal. Non Lin{\'e}aire},
  35(6):1557--1576, 2018.

\bibitem{HuangCainesMalhame06}
M.~Huang, P.~Caines, and R.~Malham{\'e}.
\newblock Large populations stochastic dynamics games: Closed-loop
  {{McKean-Vlasov}} systems and the {{Nash}} certainly equivalence principle.
\newblock 6:221--251, 2006.

\bibitem{Caines2}
M.~Huang, P.~E. Caines, and R.~P. Malham{\'e}.
\newblock Large-population cost-coupled {{LQG}} problems with nonuniform
  agents: Individual-mass behavior and decentralized {{$\epsilon$}}-{{Nash}}
  equilibria.
\newblock {\em Institute of Electrical and Electronics Engineers},
  52(9):1560--1571, 2007.

\bibitem{john-nirenberg}
F.~John and L.~Nirenberg.
\newblock On functions of bounded mean oscillation.
\newblock {\em Communications on Pure and Applied Mathematics}, 14(3):415--426,
  1961.

\bibitem{Kizilkale2014829}
A.~Kizilkale and R.~Malhame.
\newblock Collective target tracking mean field control for electric space
  heaters.
\newblock {\em 2014 22nd Mediterranean Conference on Control and Automation,
  MED 2014}, pages 829--834, 2014.

\bibitem{KizilkaleSM19}
A.~Kizilkale, R.~Salhab, and R.~Malham{\'e}.
\newblock An integral control formulation of mean field game based large scale
  coordination of loads in smart grids.
\newblock {\em Autom.}, 100:312--322, 2019.

\bibitem{ll1}
J.-M. Lasry and P.-L. Lions.
\newblock Jeux {\`a} champ moyen. {{I}}. {{Le}} cas stationnaire.
\newblock {\em Comptes Rendus Mathematique. Academie des Sciences. Paris},
  343(9):619--625, 2006.

\bibitem{ll2}
J.-M. Lasry and P.-L. Lions.
\newblock {Jeux {\`a} champ moyen. II -- Horizon fini et contr{\^o}le optimal}.
\newblock {\em Comptes Rendus. Math{\'e}matique}, 343(10):679--684, 2006.

\bibitem{lasryMeanFieldGames2007a}
J.-M. Lasry and P.-L. Lions.
\newblock Mean field games.
\newblock {\em Japanese Journal of Mathematics}, 2(1):229--260, Mar. 2007.

\bibitem{LCDF}
P.~L. Lions.
\newblock Coll{\`e}ge de {{France Course}} on {{Mean-field}} games.
\newblock 2007/.

\bibitem{lucasmoll}
R.~Lucas and B.~Moll.
\newblock Knowledge growth and the allocation of time.
\newblock {\em Journal of Political Economy}, 122(1), 2014.

\bibitem{MA2020108774}
Y.~Ma and M.~Huang.
\newblock Linear quadratic mean field games with a major player: {{The}}
  multi-scale approach.
\newblock {\em Automatica}, 113:108774, Mar. 2020.

\bibitem{moser}
J.~Moser.
\newblock On {{Harnack}}'s theorem for elliptic differential equations.
\newblock {\em Communications on Pure and Applied Mathematics}, 14(3):577--591,
  1961.

\bibitem{PV15}
E.~Pimentel and V.~Voskanyan.
\newblock Regularity for second-order stationary mean-field games.
\newblock {\em Indiana University Mathematics Journal}, 66(1):1--22, 2017.

\bibitem{PrSa2016}
A.~Prosinski and F.~Santambrogio.
\newblock Global-in-time regularity via duality for congestion-penalized {{Mean
  Field Games}}.
\newblock {\em Stochastics-an International Journal of Probability and
  Stochastic Processes}, 89(6-7):923--942, 2017.

\bibitem{serrin64}
J.~Serrin.
\newblock Local behavior of solutions of quasi-linear equations.
\newblock {\em Acta Mathematica}, 111(none):247--302, Jan. 1964.

\bibitem{Stella20132519}
L.~Stella, F.~Bagagiolo, D.~Bauso, and G.~Como.
\newblock Opinion dynamics and stubbornness through mean-field games.
\newblock {\em Proceedings of the IEEE Conference on Decision and Control},
  pages 2519--2524, 2013.

\bibitem{trudinger}
N.~S. Trudinger.
\newblock On {H}arnack type inequalities and their application to quasilinear
  elliptic equations.
\newblock {\em Communications on Pure and Applied Mathematics}, 20(4):721--747,
  1967.

\end{thebibliography}
